\documentclass{smjour}
\usepackage{nummeth,amstext,amsfonts,amssymb,amsbsy,array,graphicx,epsfig}
\usepackage{latexsym}
\usepackage{amsmath}
\newcommand{\jump}[1]{[\![#1]\!]}

\begin{document}
\title{Projection stabilisation of Lagrange multipliers for the
  imposition of constraints on interfaces and boundaries}
\titlerunninghead{Projection stabilisation
 of Lagrange multipliers}
\author{Erik Burman}
\authorrunninghead{Erik Burman}
\affil{Department of Mathematics\\ University College London\\ UK-WC1E
  6BT\\ United Kingdom \\
\tt{E.Burman@ucl.ac.uk}
}

\abstract{
Projection stabilisation applied to general Lagrange multiplier finite element methods is introduced
and analysed in an abstract framework. We then consider some
applications of the stabilised methods: (i) the weak imposition of
boundary conditions, (ii) multi-physics coupling on unfitted meshes,
(iii) a new interpretation of the classical residual stabilised Lagrange
multiplier method introduced in  {\emph{H.~J.~C.~ Barbosa and T.~J.~R.~Hughes,  The finite element method with {L}agrange multipliers on the
  boundary: circumventing the {B}abu\v ska-{B}rezzi condition. {\em Comput. Methods Appl. Mech. Engrg.}, 85(1):109--128, 1991}}
.
}

\newcommand{\sder}[1]{#1_{\beta}}
\newcommand{\cder}[1]{#1_{\beta^\perp}}
\newcommand{\Proj}{{\mathcal P}}
\newcommand{\restr}[1]{ \rule[-1ex]{0.1pt}{2.1ex}_{#1}}
\newcommand{\norm}[2]{\left\|#1\right\|_{#2}}
\newcommand{\hexagon}{
\setlength{\unitlength}{.5mm}
\begin{figure}
\begin{center}
\begin{picture}(120,100)(-5,0)
   \linethickness{0.5pt}
   \put(60,50){\circle*{3}}
   \put(70,60){\makebox(0,0)[t]{$S_0$}}
   \put(60,50){\line(1,0){40.0}}
   \put(100,50){\circle*{3}}
   \put(60,50){\line(-1,0){40.0}}
   \put(20,50){\circle*{3}} 
   \put(20,50){\line(1,2){20.0}}
   \put(40,90){\circle*{3}} 
   \put(40,90){\line(1,0){40.0}}
   \put(80,90){\circle*{3}} 
   \put(40,90){\line(1,-2){20.0}}
   \put(80,90){\line(1,-2){20.0}}
   \put(80,90){\line(-1,-2){20.0}}
   \put(20,50){\line(1,-2){20.0}}
   \put(40,10){\circle*{3}} 
   \put(40,10){\line(1,0){40.0}}
   \put(80,10){\circle*{3}} 
   \put(40,10){\line(1,2){20.0}}
   \put(80,10){\line(1,2){20.0}}
   \put(80,10){\line(-1,2){20.0}}
\end{picture}
\end{center}
\caption{Macroelement $\Omega_{S_0}$ in two space dimensions}
\label{macro}
\end{figure}
}
\newcommand{\macroelement}{
\setlength{\unitlength}{.5mm}
\begin{figure}
\begin{center}
\begin{picture}(120,100)(-5,0)
   \linethickness{0.5pt}
   \put(60,50){\circle*{3}}
   \put(70,60){\makebox(0,0)[t]{$S_0$}}
   \put(60,50){\line(1,0){40.0}}
   {\circle*{3}}
   {\line(-1,0){40.0}}
   \put(20,50){\circle*{3}} 
   \put(20,50){\line(1,2){20.0}}
   \put(40,90){\circle*{3}} 
   \put(40,90){\line(1,0){40.0}}
   \put(80,90){\circle*{3}} 
   \put(40,90){\line(1,-2){20.0}}
   \put(80,90){\line(1,-2){20.0}}
   \put(80,90){\line(-1,-2){20.0}}
   \put(20,50){\line(1,-2){20.0}}
   \put(40,10){\circle*{3}} 
   \put(40,10){\line(1,0){40.0}}
   \put(80,10){\circle*{3}} 
   \put(40,10){\line(1,2){20.0}}
   \put(80,10){\line(1,2){20.0}}
   \put(80,10){\line(-1,2){20.0}}
\end{picture}
\end{center}
\caption{Macroelement $\Omega_{S_0}$ in two space dimensions}
\label{macro}
\end{figure}
}
\newcommand{\diagram}{
\begin{figure}
\setlength{\unitlength}{0.75mm}
\begin{center}
\begin{picture}(120,180)(-10,-95)
   \put(12,0){\vector(1,0){75}}
   \put(12,-95){\vector(0,1){190}}
   \put(7,100){\makebox(0,0)[t]{$\bv$}}
   \put(90,-2){\makebox(0,0)[t]{$\bu$}}
   \put(10,90){\line(1,0){4}}
   \put(0,93){\makebox(0,0)[t]{$\frac{3 \pi}{2} - \alpha$}}
   \put(10,40){\line(1,0){4}}
   \put(0,42){\makebox(0,0)[t]{$\frac{\pi}{2} + \alpha$}}
   \put(12,90){\line(1,-1){66}}
   \put(12,40){\line(1,-1){66}}
   \put(10,65){\line(1,0){4}}
   \put(3,66){\makebox(0,0)[t]{$\pi$}}
   \put(45,-2){\line(0,1){4}}
   \put(47,-2){\makebox(0,0)[t]{$\frac{\pi}{2}$}}
   \put(77,3){\line(0,1){4}}
   \put(79,-2){\makebox(0,0)[t]{$\pi$}}
   \multiput(12,65)(4,0){7}{\line(1,0){3}}
   \put(22,75){\makebox(0,0)[t]{(ii)}}
   \multiput(45,-95)(0,4){45}{\line(0,1){3}}
   \put(27,50){\makebox(0,0)[t]{(i)}}
   \multiput(77,-95)(0,4){45}{\line(0,1){3}}
   \put(57,30){\makebox(0,0)[t]{(iii)}}
   \put(68,-4){\makebox(0,0)[t]{(iv)}}
   \put(0,0){\makebox(0,0)[t]{$0$}}
   \put(10,25){\line(1,0){4}}
   \put(0,28){\makebox(0,0)[t]{$\frac{\pi}{2} - \alpha$}}
   \put(10,-25){\line(1,0){4}}
   \put(0,-23){\makebox(0,0)[t]{$-\frac{\pi}{2} + \alpha$}}
   \put(12,25){\line(1,-1){66}}
   \put(12,-25){\line(1,-1){66}}
   \put(3,-64){\makebox(0,0)[t]{$-\pi$}}
   \multiput(10,-65)(4,0){18}{\line(1,0){3}}
   \put(20,12){\makebox(0,0)[t]{(vi)}}
   \put(27,-15){\makebox(0,0)[t]{(v)}}
   \put(57,-35){\makebox(0,0)[t]{(vii)}}
   \put(68,-69){\makebox(0,0)[t]{(viii)}}
\end{picture}
\end{center}
   \caption{The different domains in the $\bu/\bv$-plane}
   \label{anglediagram}
   \end{figure}
}

\newcommand{\angledef}{
\begin{figure}
\begin{center}
\setlength{\unitlength}{0.5mm}
\begin{picture}(120,100)(-5,0)
   \linethickness{0.1pt}
   \put(60,15){\circle*{3}}
   \put(60,15){\vector(4,1){75}}
   \put(60,15){\vector(1,2){30}}
   \put(60,15){\vector(-1,1){45}}
   \put(60,15){\arc{15}{-1.107}{-0.246}}
   \put(67.3,16.8){\line(0,1){2.5}}
   \put(67.3,16.8){\line(-7,3){2.5}}
   \put(60,15){\arc{20}{-2.36}{-1.107}}
   \put(64.4,23.8){\line(-2,4){1.2}}
   \put(64.4,23.8){\line(-4,-1){2.5}}
   \put(60,35){\makebox(0,0)[t]{$\bu$}}
   \put(75,30){\makebox(0,0)[t]{$\bv$}}
   \put(95,80){\makebox(0,0)[t]{$\beta$}}
   \put(10,70){\makebox(0,0)[t]{$\nabla u$}}
   \put(145,40){\makebox(0,0)[t]{$\nabla v_0$}}
    \end{picture}
    \caption{The definition of angles $\bu$ and $\bv$.}
    \label{angles}
\end{center}
\end{figure}
}
        
\newcommand{\meas}[1]{\mbox{m}_{#1}}
\newcommand{\minang}{\alpha_K}
\newcommand{\tf}{v_0}
\newcommand{\femu}{U}
\newcommand{\macroele}{\Omega_{S_0}}
\newcommand{\epara}{\e^K_{\parallel}}
\newcommand{\eperp}{\e^K_{\perp}}
\newcommand{\ugpara}{\nabla {\femu}_{\parallel}^K}
\newcommand{\ugperp}{\nabla {\femu}_{\perp}^K}
\newcommand{\bu}{\theta_K} 
\newcommand{\bv}{\varphi_K}
\newcommand{\bvp}{\tilde{\varphi}_K}
\newcommand{\bfA}{\mbox{\bf A}}
\newcommand{\bfB}{\mbox{\bf B}}
\newcommand{\bfC}{\mbox{\bf C}}
\newcommand{\bfd}{\mbox{\bf d}}
\newcommand{\bfD}{\mbox{\bf D}}
\newcommand{\bfe}{\mbox{\bf e}}
\newcommand{\bfE}{\mbox{\bf E}}
\newcommand{\bff}{\mbox{\bf f}}
\newcommand{\bfF}{\mbox{\bf F}}
\newcommand{\bfg}{\mbox{\bf g}}
\newcommand{\bfG}{\mbox{\bf G}}
\newcommand{\bfI}{\mbox{\bf I}}
\newcommand{\bfK}{\mbox{\bf K}}
\newcommand{\bfM}{\mbox{\bf M}}
\newcommand{\bfn}{\mbox{\bf n}}
\newcommand{\bfN}{\mbox{\bf N}}
\newcommand{\bfP}{\mbox{\bf P}}
\newcommand{\bfq}{\mbox{\bf q}}
\newcommand{\bfr}{\mbox{\bf r}}
\newcommand{\bfR}{\mbox{\bf R}}
\newcommand{\bfS}{\mbox{\bf S}}
\newcommand{\bfu}{\mbox{\bf u}}
\newcommand{\bfU}{\mbox{\bf U}}
\newcommand{\bfv}{\mbox{\bf v}}
\newcommand{\bfvt}{\bar{\mbox{\bf v}}}
\newcommand{\bfV}{\mbox{\bf V}}
\newcommand{\bfw}{\mbox{\bf w}}
\newcommand{\bfW}{\mbox{\bf W}}
\newcommand{\bfx}{\mbox{\bf x}}

\newcommand{\rh}{\varrho}

\newcommand{\bfUw}{\bar{\bf u}}
\newcommand{\Uw}{\bar{u}}

\newcommand{\bA}{\tilde{\bfA}}
\newcommand{\bS}{\tilde{\bfS}}
\newcommand{\bD}{\tilde{\bfD}}
\newcommand{\cD}{\tilde{\cal D}}

\newcommand{\dime}{d}  
\newcommand{\normal}{\mbox{\bf n}} 

\newcommand{\test}{\phi}

\newcommand{\bx}{\bar{x}}
\newcommand{\bt}{\bar{t}}

\newcommand{\un}{{\tilde{u}}}
\newcommand{\eps}{\varepsilon}
\newcommand{\St}{\hat{S}}
\newcommand{\st}{s}
\newcommand{\ps}{\mbox{\boldmath$\psi$}}
\newcommand{\PS}{{\bf \Psi}}
\newcommand{\vp}{\mbox{\boldmath$\varphi$}}
\newcommand{\vpi}{{\varphi}}
\newcommand{\VP}{{\bf \Phi}}
\newcommand{\VPI}{\Phi}
\newcommand{\ut}{{\tilde{u}}}
\newcommand{\Ut}{{\tilde{U}}}
\newcommand{\diff}{\hat{\eps}}
\newcommand{\dif}{{\cal E}}
\newcommand{\Diff}{\bfD}
\newcommand{\bDif}{\tilde{\Diff}}
\newcommand{\dera}{{D_{1}}}
\newcommand{\derb}{{D_{2}}}
\newcommand{\derbd}{{D^{\dif}_{2}}}
\newcommand{\pr}{\mbox{P}}
\newcommand{\pt}{\partial_t}
\newcommand{\px}{\partial_x}
\newcommand{\py}{\partial_y}
\newcommand{\pz}{\partial_z}
\newcommand{\id}{\mbox{I}}

\newcommand{\bbR}{{\Bbb{R}}}
\newcommand{\bbP}{{\Bbb{P}}}
\newcommand{\bfTh}{\mbox{\boldmath$\Theta$}}
\newcommand{\bfbeta}{\mbox{\boldmath$\beta$}}
\newcommand{\bfdelta}{\mbox{\boldmath$\delta$}}
\newcommand{\bftau}{\mbox{\boldmath$\tau$}}
\newcommand{\bfEPS}{{\bf \mathcal{E}}}
\newcommand{\bfxi}{\mbox{\boldmath$\xi$}}
\newcommand{\dxdt}{{\mathrm{d}}x{\mathrm{d}}t}
\newcommand{\dxds}{{\mathrm{d}}x{\mathrm{d}}s}
\newcommand{\bdxdt}{{\mathrm{d}}\bar{x}{\mathrm{d}}\bar{t}}
\newcommand{\dx}{{\mathrm{d}}x}
\newcommand{\dt}{{\mathrm{d}}t}
\newcommand{\ds}{{\mathrm{d}}s}
\newcommand{\dtau}{{\mathrm{d}}\tau}
\newcommand{\bfLam}{\mbox{\boldmath$\Lambda$}}
\newcommand{\bint}{\mbox{$\displaystyle\int_{S_n}$}}
\renewcommand{\theenumi}{(\Alph{enumi})}
\def\eref#1{(\ref{#1})}

\newcommand{\mes}{\operatorname{meas}}
\newcommand{\ang}{\operatorname{ang}}
\renewcommand{\dim}{\operatorname{dim}}

\newcommand{\cT}{\mathcal{T}}
\newcommand{\cE}{\mathcal{E}}

\newcommand{\cut}{c}
\def\IR{\mathbb R}
\begin{article}
\section{Introduction}\label{intro}
The use of Lagrange multipliers to impose constraints in the finite
element method is a well-known and powerful technique. To obtain a
stable method the finite element spaces for the primal variable and
the multiplier must be carefully matched so as to satisfy an inf-sup
condition uniformly in the mesh parameter  (see Babuska
\cite{Ba72}, Brezzi \cite{Bre74}, Pitk\"aranta \cite{Pit79},
\cite{Pit80}). If an unstable pair is used, stability can be recovered using
a stabilised method \cite{BH91, RL04}. 

In many cases such as when
imposing incompressibility for flow problems there are several choices available,
both to design inf-sup stable velocity-pressure pairs (see for
instance \cite{BF91}) and to design
stabilised methods for pairs that do not satisfy the inf-sup
condition. A class of method that has been particularly successful
recently are
projection stabilisation methods. Loosely speaking such methods ensure
stability by adding a term that penalises the difference between the
pressure solution and its projection onto some inf-sup stable space
\cite{BF01, BB01, DB04, Bu08}.

Recently there has been renewed interest in Lagrange
multiplier method in the context of imposing constraints on embedded
boundaries and multi-scale or multi-physics coupling problems
\cite{BMW09, CB09, BDR01, LB10, HR10}.  Also here
care must be taken to chose pairs of finite element spaces that
satisfy the appropriate inf-sup condition, in order to avoid spurious
oscillations or locking.

In some of these cases, although the choice of stable space is known,
it may be inconvenient. Either the spaces may be very complicated to design or use from an
implementation point of view, or the multiplier space simply is too
small to give sufficient control of 
the constraint. Here the state of the art method for stabilisation is
the residual based formulation introduced by Barbosa and Hughes
\cite{BH91}. This method has been shown to be closely related to
Nitsche's method, in cases where the Lagrange multiplier can be
eliminated locally \cite{Sten95}. It can also be applied for interface
coupling proplems, with a large flexibility in the choice of
multiplier space, see for instance \cite{HLPS05}.

It appears that the idea of projection
stabilisation, that has been very successful for Stokes' problem,
has not yet been exploited to its full potential in the context of
other type of problems featuring
Lagrange multipliers. However it appears that such an approach can give
certain advantages.
\begin{itemize}
\item For domain decomposition with non-matching
meshes it allows for the use of a Lagrange multiplier that is defined on
a third mesh which can be chosen arbitrarily (typically
structured). In this case the stabilisation operator only acts on the
multiplier space, see \cite{BH10b}. This reduces the problem of interpolating between two fully
unstructured meshes to that of interpolating from two unstructured
meshes to one structured mesh.

\item Another example is fictitious domain methods where the multiplier can
be chosen piecewise constant per element and distributed in the
interface zone if projection stabilisation is used \cite{BH10a}. This choice is advantageous from the point of
view of implementation, but normally prohibited since the inf-sup condition
fails \cite{GG95}.

\item Compared to Nitsche type methods or the Barbosa-Hughes
  stabilised method the projection stabilised multiplier method does
  not use the trace of the stress tensor explicitly. This is
particularly  advantageous in the nonlinear case, since the
nonlinearity then appears only in the bulk and not in the interface terms.
\end{itemize}

Stabilised Lagrange methods seem to be attracting increasing attention, in
particular for the imposition of embedded Dirichlet boundary
conditions \cite{CB09, GW10, BH10a, HCD11, BC11}. It is interesting to
note that the extension to XFEM type interface coupling methods is
practically always straightforward.

The focus of the present paper is on the generality of this type of
method. We prove a wellposedness result for discrete solutions and a best approximation
result in an abstract framework. Then we show how to apply the ideas
to the analysis and design 
of stabilised Lagrange multiplicator methods
 first in the simple case of the weak imposition of boundary
 conditions and then sketching an unfitted finite element method for
 multi-physics coupling.


As a last example of the applicability of our framework we give a new
interpretation of the non-symmetric version of the method of Barbosa \& Hughes \cite{BH91}, for the imposition of boundary
conditions. In these methods, the stabilisation acts on the difference
between the multiplier and the  gradient of the primal
variable. Using a recent stability result for the penalty-free,
nonsymmetric Nitsche's method \cite{Bu11b}, we show that the
nonsymmetric version of such stabilised
Lagrange multiplier methods are in fact closely related to projection
stabilisation methods by the inf-sup stability of the Lagrange multiplier
space consisting of normal gradients of the primal variable on the boundary trace mesh.

 As a model problem the reader may consider the Poisson problem set on an open connected domain
$\Omega \subset \mathbb{R}^d$, $d=2,3$, with polygonal (or polyhedral)
boundary. Find $u:\Omega \rightarrow \mathbb{R}$ such that
\begin{equation}\label{Poisson}
\begin{split}
-\Delta u & = f \mbox{ in } \Omega\\
u & = 0 \mbox{ on } \partial \Omega.
\end{split}
\end{equation}
The weak formulation of this problem, using Lagrange multipliers to
impose the boundary constraints, takes the following form: find
$(u,\lambda) \in H^1(\Omega) \times H^{-\frac12}(\partial \Omega)$
such that
\begin{multline}\label{weak_Poisson}
\int_\Omega \nabla u \cdot \nabla v ~\mbox{d} x + \int_{\partial
  \Omega} \lambda v  ~\mbox{d} s + \int_{\partial
  \Omega} \mu u ~\mbox{d} s = \int_\Omega f v ~ \mbox{d}x\\
\forall (v,\mu) \in H^1(\Omega) \times H^{-\frac12}(\partial \Omega).
\end{multline}
We will frequently use the notation $a \lesssim b$ for $a \leq C b$
where $C$ is a constant independent of the mesh-size, but not
necessarily of the local mesh geometry. We also assume
quasi-uniformity and shape regularity for all meshes.

\section{Abstract setting}\label{abstractset}
We will here give
an abstract framework for this type of method to give some understanding of the underlying idea. Our aim is to make the simplest possible framework. Let 
$$a(\cdot,\cdot):V\times V \rightarrow \mathbb{R}$$
 and 
$$
b(\cdot,\cdot):L \times V \rightarrow \mathbb{R}
$$ 
be two bilinear forms representing the partial differential operator on weak form and the constraint respectively. 
The abstract formulation 
then writes: find $(u,\lambda) \in V \times L$ such that
\begin{equation}\label{const_cont}
a(u,v)+b(\lambda,v)+b(\mu,u)= (f,v)
\end{equation}
for all $(v,\mu) \in V\times L$. We assume that the spaces $V$ and
$L$ are chosen such that the problem is well posed.
Firstly we assume that the bilinear forms satisfy the following continuities
\[
a(u,v) \lesssim \|u\|_V \|v\|_V,\quad \forall u,v \in V
\]
\[
b(\lambda,v) \lesssim \|\lambda\|_L \|v\|_V,\quad \forall \lambda \in
L \mbox{ and } \forall u\in V
\]
and secondly that the form $a(u,v)$ is coercive on the kernel of $b(\lambda,v)$, i.e. 
\[
\|v\|^2_V \lesssim a(v,v),\, \mbox{ for all } v \mbox{ such that } b(\mu,v)=0,\quad \forall \mu \in L.
\]
Finally we assume that the Babuska-Brezzi condition is satisfied so that $\forall \lambda \in L$ there holds
\[
\|\lambda\|_L \lesssim \sup_{v \in V} \frac{b(\lambda,v)}{\|v\|_V}.
\]
\begin{example}
In the case of the Poisson problem \eqref{Poisson} above the bilinear
forms are given by the weak formulation \eqref{weak_Poisson} as
\begin{equation}\label{ex_a}
a(u,v):=\int_\Omega \nabla u \cdot \nabla v ~\mbox{d} x,
\end{equation}
and
\begin{equation}\label{ex_b}
b(\lambda,v):=\int_{\partial
  \Omega} \lambda v  ~\mbox{d} s.
\end{equation}
The spaces are given by $V:=H^1(\Omega)$ and $L:=
H^{-\frac12}(\partial \Omega)$.
\end{example}
Now consider the discretisation of the problem \eqref{const_cont} in
$V_h \subset V, L_h \subset L$. We assume that these spaces satisfy the discrete version of the inf-sup condition
uniformly so that $\forall \lambda_h \in L_h$ there holds
\begin{equation}\label{brezzi_disc}
\|\lambda_h\|_{L} \lesssim \sup_{v_h \in V_h} \frac{b(\lambda_h,v_h)}{\|v_h\|_{V}}.
\end{equation}

It is known \cite{For77, EG04} that the discrete inf-sup condition is equivalent to the
existence of an interpolant $\pi_F:V \rightarrow V_h$ such that for
any 
$ v \in V$ there holds 
\begin{equation}\label{fortin}
b(v - \pi_F v, \mu_h) = 0 \quad \forall \mu_h \in L_h,\quad \mbox{ and
} \|\pi_F v\|_V \lesssim \| v\|_V.
\end{equation}
We introduce norms defined on functions in
the discrete spaces $\|\cdot\|_{L_h}$  and $\|\cdot\|_{V_h}$ and
assume that the bilinear forms also satisfy the following continuities,
\[
a(u_h,v_h) \leq \|u_h\|_{V_h} \|v_h\|_{V_h},\quad \forall u_h,v_h \in V_h
\]
\[
b(\lambda_h,v_h) \leq \|\lambda_h\|_{L_h} \|v_h\|_{V_h},\quad \forall \lambda_h \in
L_h\mbox{ and } \forall u_h\in V_h.
\]
We will also assume that $\|v_h\|_V \lesssim \|v_h\|_{V_h}$ for all
$v_h \in V_h$. 
\begin{example}
For $V_h \subset H^1(\Omega)$ and $L_h \subset H^{-\frac12}(\partial \Omega)
\cap L^2(\partial \Omega))$ we may take
\[
\|\mu_h\|_{L_h}:= \|h^{\frac12}
\mu_h\|_{L^2(\partial \Omega)}
\]
and
\[
\|v_h\|_{V_h}:= \|\nabla v_h\|_{L^2(\Omega)} + \|h^{-\frac12}
v_h\|_{L^2(\partial \Omega)}.
\]
It follows immediately by the Cauchy-Schwarz inequality that the following continuities hold
\begin{equation}\label{acontdisc}
a(u_h,v_h) \lesssim \|u_h\|_{V_h} \|v_h\|_{V_h} \quad \forall u_h,v_h
\in V_h
\end{equation}
and
\begin{equation}\label{bcontdisc}
b(\lambda_h,v_h) \lesssim \|\lambda_h\|_{L_h} \|v_h\|_{V_h} \quad \forall
\lambda_h \in L_h , \, \forall v_h \in V_h.
\end{equation}
\end{example}

This leads to the following formulation: find $\{u_h,\lambda_h\} \in V_h \times L_h$ such that
\begin{equation}\label{fem1}
a(u_h,v_h)+b(\lambda_h,v_h)-b(u_h,\mu_h)= (f,v_h), \quad\forall \{v_h,\mu_h\} \in V_h \times L_h.
\end{equation}
Then we know that the discrete problem is well posed and we may prove optimal
convergence provided the spaces have optimal approximation
properties. 
We will denote the kernel of $b(\cdot,\cdot)$
by 
$$
K_h := \{v_h \in V_h : b(\mu_h,v_h) = 0,\, \forall \mu_h \in L_h \}.
$$

Consider now the case where we do not want to use the space $L_h$
because it leads to inconvenient interpolation problems. We want to work with the possibly completely
unrelated, richer, space $\Lambda_h$, for which no stability is known
to hold, but which is convenient from the point of view of
implementation. We also assume that there exists a projection $\pi_L:
\Lambda_h \rightarrow L_h$ so that the following continuity holds for
all $v \in V$,
\begin{equation}\label{stab_cond_pil}
b(\lambda_h - \pi_L \lambda_h, v) \lesssim \|\lambda_h - \pi_L
\lambda_h\|_{L_h} \|v\|_V.
\end{equation}
This is a technical assumption that only constrains the choice of
$\pi_L$ used in the analysis and not in practice, as we shall see
later. When the Fortin interpolant is used for the analysis as we do
here this assumption
is convenient since otherwise one must work in the norm $\|\cdot\|_L$
when designing the stabilisation term. Under \eqref{stab_cond_pil} one
may use
the discrete norm directly. An alternative route for the analysis is
to use a discrete inf-sup condition in the discrete norm and
associated analysis.

Instead of \eqref{brezzi_disc} we then have the
following stability property.
\begin{lemma}\label{stab_infsup}
For all $\lambda_h \in \Lambda_h$ there holds
\[
\|\lambda_h\|_{L} \lesssim \sup_{v_h \in V_h} \frac{b(\lambda_h,v_h)}{\|v_h\|_{V}} + \|\lambda_h - \pi_L \lambda_h\|_{L_h},
\]
where $\pi_L:\Lambda_h \rightarrow L_h$ denotes an interpolation operator
from $\Lambda_h$ to $L_h$ such that \eqref{stab_cond_pil} holds.
\end{lemma}
\begin{proof}
By the continuous inf-sup condition there holds for all $\lambda_h \in \Lambda_h$,
\[
\|\lambda_h\|_{L} \lesssim \sup_{v \in V} \frac{b(\lambda_h,v)}{\|v\|_{V}}.
\]
Since $\pi_L \lambda_h \in L_h$ the condition \eqref{brezzi_disc}
holds and hence by \eqref{stab_cond_pil}
\begin{multline*}
\|\lambda_h\|_{L} \lesssim \sup_{v \in V} \frac{b(\lambda_h - \pi_L
  \lambda_h,v) + b(\pi_L
  \lambda_h,\pi_F v)}{\|v\|_{V}} 
\lesssim \|\lambda_h - \pi_L \lambda_h\|_{L_h} +\frac{b(\pi_L
  \lambda_h,\pi_F v)}{\|\pi_F v\|_{V}}.
\end{multline*}
We may then add and subtract $\lambda_h$ in the last term in the right
hand side to obtain using \eqref{stab_cond_pil}
\[
\frac{b(\pi_L
  \lambda_h,\pi_F v)}{\|\pi_F v\|_{V}} = \frac{b(
  \lambda_h,\pi_F v)+b(\pi_L
  \lambda_h - \lambda_h,\pi_F v)}{\|\pi_F v\|_{V}} \lesssim
\|\lambda_h - \pi_L \lambda_h\|_{L_h} + \sup_{v_h\in V_h} \frac{b(\lambda_h,v_h)}{\|v_h\|_{V}}.
\]
\end{proof}\\
This means that, provided that we can control the distance $\|\lambda_h - \pi_L \lambda_h\|_{L_h}$  from the approximation in the space $\Lambda_h$ to the space $L_h$, which satisfies the LBB-condition, we will have stability using the space $\Lambda_h$. The simplest way of obtaining this is to add a symmetric operator $s(\lambda_h,\mu_h)$,
designed so that 
\begin{equation}\label{stab_equiv}
\|\lambda_h - \pi_L \lambda_h\|_{L_h}^2\lesssim s(\lambda_h,\lambda_h) 
\end{equation} 
to the formulation \eqref{fem1}.
Since the effect of $s(\cdot,\cdot)$ is to reduce the effective
dimension of the space $\Lambda_h$ it 
can be thought of as a {\emph{coarsening operator}}.

This leads to the stabilised formulation:
\begin{multline}\label{stab_fem}
a(u_h,v_h) + b(\lambda_h,v_h) + b(\mu_h,u_h) - s(\lambda_h,\mu_h) = (f,v_h)\, \\ \mbox{ for all } (v_h,\mu_h) \in V_h \times \Lambda_h.
\end{multline}
The signs in \eqref{stab_fem} have been chosen so as to preserve
symmetry, note however that the problem is indefinite due to the
saddle point structure.
 For the operator $s(\cdot,\cdot)$, the following design
 criteria are advantageous:
\begin{itemize}\item minimal dependence of the stable subspace $L_h$
\item the smallest possible stencil
\item optimal weak consistency. 
\end{itemize}
Often $s(\cdot,\cdot)$ may be chosen as the jump 
of the function or of function derivatives over element faces in the
multiplier space and we will explore this possibility further below. 

When we work with the
multiplier space $\Lambda_h$, it is no longer sufficient to assume that $a(u_h,v_h)$ is
coercive on the kernel $K_h$ of $b(\lambda_h,v_h)$, for $\lambda_h\in \Lambda_h$. 
Indeed the stabilisation term could upset the coercivity.
To ensure that the constraint remains strong enough compared to the
penalty term we assume that for all
$u_h \in V_h$ there exists $\xi_h(u_h) \in \Lambda_h$ such that
\begin{equation}\label{coercivity_assumption}
\begin{array}{l}
\alpha_\xi \|u_h\|_{V_h}^2 \leq  a(u_h,u_h) +  b(\xi_h(u_h),u_h)\\[3mm]
s(\xi_h(u_h),\xi_h(u_h))^{\frac12} \leq c_s \|u_h\|_{V_h},
\end{array}
\end{equation}
where $c_s$ can be made small by choosing the stabilisation parameter small.
$\xi_h(u_h)$ is related to the constraint that one wishes to
impose. For the case of weak boundary conditions typically
$\xi_h(u_h)$ is the projection of the trace of $u_h$ onto the Lagrange
multiplier space as we shall see later. We first state and prove the
obtained coercivity result in a lemma and then conclude this section
by our main theorem, showing a best approximation
property for the formulation \eqref{stab_fem}.

\begin{lemma}\label{pos_lemma}
For all $\{u_h,\lambda_h\} \in V_h \times L_h$ there holds 
\begin{equation}\label{positivity0}
\|u_h\|_{V_h}^2 + s(\lambda_h,\lambda_h) \lesssim a(u_h,u_h) +
b(\lambda_h,u_h) - b(\lambda_h - \xi_h(u_h),u_h) + s(\lambda_h,
\lambda_h - \xi_h(u_h)).
\end{equation}
\end{lemma}
\begin{proof}
Starting from the right hand side of \eqref{positivity0} we have using
\eqref{coercivity_assumption} and an arithmetic-geometric inequality
\begin{multline*}
a(u_h,u_h) +
b(\lambda_h,u_h) - b(\lambda_h - \xi_h(u_h),u_h) + s(\lambda_h,
\lambda_h - \xi_h(u_h)) \\ 
\ge \alpha_\xi \|u_h\|_{V_h}^2 + \frac12 s(\lambda_h,
\lambda_h) - \frac12 s(\xi_h(u_h),\xi_h(u_h)).
\end{multline*}
Using now the second inequality of \eqref{coercivity_assumption} we may
conclude, assuming $c_s$ small enough.
\end{proof}
\begin{remark}
If $\xi_h(u_h)$ may be chosen such that $s(\xi_h(u_h),\nu_h)=0$,
$\forall \nu_h \in \Lambda_h$ then
\eqref{positivity0} holds without constraints on $c_s$.
\end{remark}
\begin{theorem}\label{best_approx_stab}
Assume that the coercivity condition \eqref{coercivity_assumption} holds for $V_h \times \Lambda_h$ and that there exists a space $L_h$ such that the condition \eqref{brezzi_disc} holds for the pair $V_h\times L_h$.

Then the system \eqref{stab_fem} admits a unique solution
$\{u_h,\lambda_h\}$. This solution satisfies the following best
approximation property 
\[
 \|u - u_h\|_{V} + \|\lambda - \lambda_h\|_{L} \lesssim \inf_{y_h \in V_h} \|u-y_h\|_{V}+\inf_{\nu_h \in \Lambda_h} (\|\lambda - \nu_h\|_{L} + s(\nu_h,\nu_h)^{\frac12}). 
\]
\end{theorem}
\begin{proof}
Assume that $u_h$ and $\lambda_h$ exist. Now by the triangular inequality
\[
\|u - u_h\|_{V} \leq \|u - \pi_F u\|_{V}+\|\pi_F u - u_h\|_{V_h},
\]
where $\pi_F$ is the Fortin interpolant associated to the spaces $V_h\times L_h$.
Set $\eta_h = u_h - \pi_F u$ and $\zeta_h =  \lambda_h - \nu_h$. By
Lemma \ref{pos_lemma} we have
\begin{multline}\label{positivity}
\|\eta_h\|_{V_h}^2 + s(\zeta_h,\zeta_h) \\ \lesssim a(\eta_h,\eta_h) + b(\zeta_h,\eta_h) 
-b(\zeta_h-\xi_h(\eta_h),\eta_h)  + s(\zeta_h,\zeta_h-\xi_h(\eta_h)).
\end{multline}
Subtracting \eqref{stab_fem} from \eqref{const_cont} with $v = v_h$,
$\mu = \mu_h$,  gives the Galerkin 
orthogonality
\begin{equation}\label{gal_ortho_const}
a(u - u_h,v_h) + b(\lambda - \lambda_h,v_h)+b(\mu_h,u-u_h)+s(\lambda_h,\mu_h) = 0.
\end{equation}
Taking $v_h = \eta_h$ and $\mu_h = -(\zeta_h-\xi_h(\eta_h))$ in \eqref{gal_ortho_const} and adding the 
left hand side of \eqref{gal_ortho_const} to the right hand side of \eqref{positivity} yields
\begin{multline}
\|\eta_h\|_{V_h}^2 + s(\zeta_h,\zeta_h) \lesssim a(u - \pi_F u,\eta_h) + b(\lambda - \nu_h,\eta_h)\\-b(\zeta_h-\xi_h(\eta_h),u-\pi_F u)+s(\nu_h,\zeta_h-\xi_h(\eta_h)).
\end{multline}
 Since $b(\mu_h,u-\pi_F u) = 0$ for all $\mu_h \in L_h$ there holds
\begin{multline*}
\|\eta_h\|_{V_h}^2 + s(\zeta_h,\zeta_h) \lesssim a(u - \pi_F u,\eta_h)
+ b(\lambda - \nu_h,\eta_h) + b(\xi_h(\eta_h) - \pi_L \xi_h(\eta_h) ,u-\pi_F u)\\
-b(\zeta_h-\pi_L \zeta_h,u-\pi_F u)+s(\nu_h,\zeta_h-\xi_h(\eta_h)).
\end{multline*}
Using the continuity \eqref{stab_cond_pil} we have
\begin{multline*}
 b(\xi_h(\eta_h) - \pi_L \xi_h(\eta_h) ,u-\pi_F u) -b(\zeta_h-\pi_L
 \zeta_h,u-\pi_F u) \\ \lesssim  (\|\xi_h(\eta_h) - \pi_L
 \xi_h(\eta_h)\|_{L_h} + \|\zeta_h - \pi_L \zeta_h\|_{L_h}) \|u -
 \pi_F u\|_V
\end{multline*}
and together with the continuity of $a(\cdot,\cdot)$ and
$b(\cdot,\cdot)$ and the bound $\|\eta_h\|_V \lesssim
\|\eta_h\|_{V_h}$ this leads to
\begin{multline}
\|\eta_h\|_{V_h}^2 + s(\zeta_h,\zeta_h) \lesssim (\|u - \pi_F
u\|_{V}+\|\lambda-\nu_h\|_{L})\|\eta_h\|_{V_h} \\
 +\|u - \pi_F u\|_{V} (\|\xi_h(\eta_h) - \pi_L
 \xi_h(\eta_h)\|_{L_h} + \|\zeta_h - \pi_L \zeta_h\|_{L_h}) \\
+s(\nu_h,\nu_h)^{\frac12}(s(\zeta_h,\zeta_h)^{\frac12}+s(\xi_h(\eta_h),\xi_h(\eta_h))^{\frac12}).
\end{multline}
Using the upper bound $\|\zeta_h-\pi_L \zeta_h\|^2_{L_h} \lesssim
s(\zeta_h,\zeta_h)$ of \eqref{stab_equiv}
and \eqref{stab_equiv} combined with the second relation of
\eqref{coercivity_assumption} to obtain $$\|\xi_h(\eta_h) - \pi_L
 \xi_h(\eta_h)\|_{L_h} \lesssim
 s(\xi_h(\eta_h),\xi_h(\eta_h))^{\frac12} \lesssim \|\eta_h\|_{V_h}$$
 we observe that
\begin{multline}\label{etabound}
\|\eta_h\|_{V_h}^2 + s(\zeta_h,\zeta_h) \lesssim (\|u - \pi_F u\|_{V}+\|\lambda-\nu_h\|_{L}+s(\nu_h,\nu_h)^{\frac12})\\
\times (\|\eta_h\|_{V_h}^2 + s(\zeta_h,\zeta_h))^\frac12.
\end{multline}
This gives the following upper bound for $\|\eta_h\|_{V_h}$
\[
\|\eta_h\|_{V_h} + s(\zeta_h,\zeta_h)^\frac12 \lesssim  \|u - \pi_F u\|_{V}+ \inf_{\nu_h \in \Lambda_h} (\|\lambda-\nu_h\|_{L}+s(\nu_h,\nu_h)^{\frac12}).
\]
By the stability of $\pi_F$ we have, for $v_h \in V_h$
\begin{multline}\label{pifstab}
\|u - \pi_F u\|_{V} \leq \|u - v_h\|_{V}+\|v_h - \pi_F u\|_{V}\\
 = \|u - v_h\|_{V}+\|\pi_F (v_h -  u)\|_{V}\lesssim \|u - v_h\|_{V}.
\end{multline}
We conclude that
\[
\|u-u_h\|_{V} + s(\zeta_h,\zeta_h)^\frac12 \lesssim  \inf_{v_h \in V_h} \|u - v_h\|_{V}+ \inf_{\nu_h \in \Lambda_h} (\|\lambda-\nu_h\|_{L}+s(\nu_h,\nu_h)^{\frac12}).
\]
For the bound on $\lambda - \lambda_h$ we use the triangle inequality to
write
\[
\|\lambda - \lambda_h\|_L \leq \|\lambda - \nu_h\|_L + \|\zeta_h\|_L
\]
followed by the the result of Lemma \ref{stab_infsup}:
\[
\|\zeta_h\|_{L} \lesssim \sup_{v_h \in V_h} \frac{b(\zeta_h,v_h)}{\|v_h\|_{V}} + \|\zeta_h - \pi_L \zeta_h\|_{L_h} \lesssim \sup_{v_h \in V_h} \frac{b(\zeta_h,v_h)}{\|v_h\|_{V}} + s(\zeta_h,\zeta_h)^\frac12.
\]
Since we already have the desired bound for the stabilisation term we
only need to consider the first term of the right hand side. By the Galerkin orthogonality
\eqref{gal_ortho_const}, with $\mu_h=0$ and the continuities of the
bilinear forms we have
\[
b(\zeta_h,v_h) = b(\lambda - \nu_h,v_h) + a(u - u_h,v_h) \lesssim (\|\lambda - \nu_h\|_{L}+\|u-u_h\|_{V}) \|v_h\|_{V}.
\]
We deduce the upper bound on $\|\zeta_h\|_{L}$,
\begin{equation}\label{zetabound}
\|\zeta_h\|_{L} \lesssim \|\lambda - \nu_h\|_{L}+ \|u-u_h\|_{V}+ s(\zeta_h,\zeta_h)^\frac12.
\end{equation}
This concludes the best approximation result.

To prove the existence and uniqueness of the discrete solution $u_h$ and $\lambda_h$ set $f=0$ and then prove that
$u_h=0$ and $\lambda_h=0$, implying that the system matrix is regular. Since
the continuous problem \eqref{const_cont} is well posed $u=0$ and $\lambda=0$,
but then by choosing $y_h=0$ and $\nu_h=0$ we have
\[
\| u_h\|_{V} + \| \lambda_h\|_{L} \lesssim \inf_{y_h \in V_h} \|v_h\|_{V}+\inf_{\nu_h \in \Lambda_h} \| \nu_h\|_{L} = 0.
\]
\end{proof}
\section{Stabilisation using jump penalty operators}\label{projpenal}
The design of the stabilisation operator $s(\cdot,\cdot)$ is
important, indeed if the construction of the projection $\pi_L$
requires a too detailed understanding of the inf-sup stable space $L_h$ the
advantages of the stabilised method may be lost. Typically this is the
case if $\pi_L$ is chosen to be the $L^2$-projection. Fortunately there are
some operators that can handle a relatively large set of spaces
$\{L_h,\Lambda_h\}$. The two natural choices are local projection
stabilisation
or interior penalty stabilisation. Herein we will only discuss the
second choice. For examples of local projection stabilisation methods
that can be used in this context we refer to \cite{BC11, Bu11b,BBH11b} where such methods have been proposed in a different
context. The extension to the present case is straightforward. Below
we will focus on the construction relevant for weak imposition of
boundary conditions. The extension to domain decomposition is
straightforward. We assume that $b(\cdot,\cdot)$ is defined by
\eqref{ex_a}.

We consider only one side $\Gamma$ of the polygonal boundary $\partial
\Omega$. Denote the trace mesh of $V_h$ by $\Gamma_{V}$.
Let the space $L_h$ be defined on a trace mesh $\Gamma_L$,
\[
L_h := \{ l_h \in L^2(\Gamma_L): l_h \vert_K \in P_k(K), \quad \forall
K \in \Gamma_L \},
\]
 and
$\Lambda_h$ on a trace mesh $\Gamma_\Lambda$,
\[
\Lambda_h := \{ \lambda_h \in L^2(\Gamma_\Lambda): \lambda_h \vert_K \in P_k(K), \quad \forall
K \in \Gamma_\Lambda \}.
\]

 The mesh function on
the trace meshes will be denoted  $h_{\Gamma,X}$, with $X = V,\, L$ or
$\Lambda$. We assume that there are positive constants $c_1$, $c_2$
and $c_3$ such
that
\[
h_{\Gamma,\Lambda}(x)  \leq c_1 h_{\Gamma,V}(x) \leq c_2 h_{\Gamma,L}(x) \leq c_3
h_{\Gamma,\Lambda}(x), \mbox{ for all } x \in \Gamma.
\]
We first note that using these spaces it is staightforward to design
$\pi_L$ so that \eqref{stab_cond_pil} holds, the only requirement is
orthogonality against constants on the elements of $\Gamma_L$. Indeed
if $\pi_L$ is chosen as the $L^2$-projection on $L_h$ it follows from
the definition of $\|\cdot\|_{L_h}$ that it can be replaced by any
interpolant in $L_h$ using the stability of the $L^2$-projection and
the quasi uniformity constraint on the mesh parameter
\[
\|\lambda_h - \pi_L \lambda_h\|^2_{L_h} \leq \sum_{K \in \Gamma_L} c_2
h_{\Gamma,L}\vert_K \|\lambda_h - \pi_L \lambda_h\|^2_{L^2(K)} =  \inf_{v_h \in L_h} \sum_{K \in \Gamma_L} c_2
h_{\Gamma,L}\vert_K  \|\lambda_h - v_h\|^2_{L^2(K)}.
\]
It follows that $v_h$ may be chosen as any interpolation of
$\lambda_h$ in $L_h$.
For imposition of boundary conditions and more generally for
domain decomposition methods the classical condition for inf-sup
stability is that $h_{\Gamma,L} > C h_{\Gamma,V}$ for some constant
$C>1$, (see \cite{Ba72}).
Using the projection stabilisation this
condition may be relaxed for the space $\Lambda_h$, since the stabilisation controls the
unstable modes. The relative difference in mesh size should be
accounted for in the stabilisation parameter to tune the constant of
the inf-sup condition. Numerical evidence however indicate that this
dependence is relatively weak. Assume for simplicity that $h_{\Gamma,\Lambda} <
h_{\Gamma,L}$. Let the interpolation operator $\tilde \pi_L:\Lambda_h \rightarrow L_h$ denote the quasi interpolation operator such
that for $u_h \in \Lambda_h$ 
\[
\tilde \pi_L u_h(x_i) := 
N_x^{-1} \sum_{\{K\in \Gamma_\Lambda: x_i \in  K\}} u_h(x_i) \vert_{K}, \mbox{
  for all nodes $x_i$ of $\Gamma_L$},
\]
where $N_x$ denotes the cardinality of the set $\{K\in \Gamma_{\Lambda}: x \in  K\}$.
Now consider any element in the trace mesh $K_\Gamma \in \Gamma_L$ and map it to a
reference element $\hat K_{\Gamma_L}$. Also map the subset 
$\Gamma'_\Lambda$ for which $\Gamma'_\Lambda :=\{ K': K' \cap
 K_\Gamma \ne \emptyset\}$ and denote the interior faces of
 $\hat \Gamma'_\Lambda$ by $\mathcal{F}'$. 
It then follows that
\[
\|\hat \lambda_h - \tilde \pi_L \hat \lambda_h\|_{L_h,\hat K_\Gamma}
\leq \sum_{\hat F \in  \mathcal{F}'}
\sum_{i=0}^k \int_{\hat F} \jump{\hat \partial_n^i \hat\lambda_h}^2
\mbox{d}\hat s,
\]
where $\jump{x}$ denotes the jump of the quantity $x$ over an element face,
with $\jump{x}=0$ for faces on the boundary and $\partial_n^i$ denotes the
normal derivative of order $i$, with $n$ the outward pointing normal
from the element $K'$ and with $\partial_n^0$, defined to be the
identity.

This upper bound on the reference element follows by the observation
that if the jump of $\hat \lambda_h$ and all its normal derivatives are
zero, then $\lambda_h$ is a polynomial over all of $\hat K_\Gamma$,
but since $\tilde \pi_L$ interpolates this polynomial $(\lambda_h -
\tilde \pi_L
\lambda_h)\vert_{\hat K_\Gamma} \equiv 0$. To show that $\hat \lambda_h$ is
a polynomial over all of $\hat K_\Gamma$ it is enough to consider one
face $\hat F$ and the associated elements such that $\hat F = \hat K_1 \cap
\hat K_2$. We choose the coordinate system so that $\hat F \subset \{ (\hat
x, \hat y): \hat y = 0\}$. We let $\hat p_i(\hat x,\hat y) = \hat
\lambda_h\vert_{\hat K_i}$ and
define the polynomial $\delta p_F(x,y) = \hat p_1 - \hat p_2$ on $\hat
K_1
\cup \hat K_2$. We must then show that 
$$
\left.
\begin{array}{r}
\delta p_F(\hat x,\hat y)\vert_{y=0} = 0, \, \forall
  \hat x \in \hat F\\[3mm]
\partial_{\hat y}^i \delta
  p_F(\hat x,\hat y)\vert_{\hat y=0} = 0, 
\, i= 1,...,k,\, \forall
  \hat x \in \hat F \end{array} \right\} \longrightarrow \delta
  p_F(\hat x,\hat y) \equiv 0.
$$
This is straightforward by noting that a polynomial of order $k$ has
$(k+1)(k+2)/2$ degrees of freedom and that $\delta p_F(\hat x,\hat y)\vert_{y=0}
= 0$ implies $k+1$ independent equations and that each $i$ $\partial_{\hat y}^i \delta
  p_F(\hat x,\hat y)\vert_{y=0} = 0$ gives $k-i+1$ independent
  equations. Summing up the independent equations we get
\[
k+1 + \sum_{i=1}^{k} (k-i+1) = \frac{(k+1)(k+2)}{2}
\]
and we conclude that $\delta p_F(\hat x,\hat y) \equiv 0$. It follows
that $\hat \lambda_h$ is defined by one polynomial over $\hat K_1 \cup
\hat K_2$. The
result on $\hat K_\Gamma$ is obtained by repeating the argument for
all faces $\hat F \in \mathcal{F}'$.

After scaling back to physical space and summing over all elements in
$\Gamma_L$ we obtain, if $\mathcal{F}_{\Lambda}$ denotes the set of
interior faces in $\Gamma_{\Lambda}$,
\[
\| \lambda_h - \tilde \pi_L \lambda_h\|^2_{L_h} \leq \sum_{ F \in \mathcal{F}_{\Lambda}}
\sum_{i=0}^k \int_{F} h^{s_0+2i} \jump{\partial^i_n \lambda_h}^2\,
\mbox{d} s.
\]
 The order $s_0$ depends on $L_h$ and follows from the scaling
 argument, in our example where the $L_h$-norm is the $h^{\frac12}$-weighted
 $L^2$-norm over $\Gamma$ we have $s_0
 = 2$. We conclude that the interior penalty stabilisation
 operator may be written
\[
s(\lambda_h,\mu_h) := \sum_{ F \in \mathcal{F}_{\Lambda}}
\sum_{i=0}^k \int_{F} h^{s_0+2i} \jump{\partial^i_n \lambda_h}\jump{\partial^i_n \mu_h}\,
\mbox{d} s.
\]

It may be inconvenient to compute all normal
derivatives up to polynomial order and an equivalent local projection approach may
be used instead as suggested in the references given above. Observe
that above we have assumed that $\Lambda_h$ and $L_h$ have the same
polynomial everywhere in the domain. If this is not the case the
analysis has to be modified accordingly.
\section{Penalty stabilisation of Lagrange multiplier formulations:
  applications} \label{sec:applications}
As an example of the above theory we recall a stabilised method introduced as a
fictitious domain method in \cite{BH10b} and using the results of \cite{GG95}
for the underlying stable
spaces. Here we will first present
the method in the simple case of weak imposition of boundary condition
and then propose an extension to unfitted finite element methods.
Both cases are considered in two space dimensions, but the extension
to three dimensions is straightforward.
\subsection{Weak imposition of boundary conditions}

In this section we will consider the problem \eqref{Poisson}, that we recall
here for the readers convenience.

Let $\Omega$ be a bounded domain in $\IR^2$, with polygonal
boundary $\partial \Omega$.
The Poisson equation that we propose as a model problem is given
by
\begin{equation}\label{strongPoisson}
\begin{array}{rcl}
-\Delta u &=& f\quad \mbox{ in } \Omega,\\
u&=&g \quad \mbox{ on } \partial \Omega,
\end{array}
\end{equation}
where $\partial \Omega$ denotes the boundary of the domain $\Omega$, $f\in
L^2(\Omega)$ and $g\in H^{\frac12}(\partial \Omega)$. Under these
assumptions \eqref{strongPoisson} has a unique solution $u\in
H^1(\Omega)$ satisfying $\|u\|_{H^1(\Omega)} \lesssim
\|f\|_{L^2(\Omega)}$. As already suggested we define $V:= H^1(\Omega)$ and
$L:=H^{-\frac12}(\partial \Omega)$.

The usual $L^2$-scalar product on the domain $\Omega$ will be denoted
by $(\cdot,\cdot)_{\Omega}$ or on the boundary
$\left<\cdot,\cdot\right>_{\partial \Omega}$. We also introduce the discrete norms 
$$\|\lambda\|^2_{\frac12,h,\partial \Omega} = \left<h^{-1}\lambda,\lambda\right>_{\partial \Omega}
,\quad \|\lambda\|^2_{-\frac12,h,\partial \Omega} =
\left<h\lambda,\lambda\right>_{\partial \Omega}$$
and
$$
\|u\|^2_{1,h}:= \|\nabla u\|^2_{L^2(\Omega)} + \|u\|^2_{\frac12,h,\partial \Omega}.
$$
 Recall that $\forall \lambda,\mu \in L^2(\partial \Omega)$ there
holds 
\begin{equation}\label{cauchyschwarz}
\left<\lambda,\mu\right>_{\partial \Omega}\leq \|\lambda\|_{-\frac12,h,\partial \Omega} \|\mu\|_{\frac12,h,\partial \Omega}.
\end{equation}
The weak formulation of the problem is given by \eqref{weak_Poisson} with $a(\cdot,\cdot)$ defined by \eqref{ex_a} and
$b(\cdot,\cdot)$ by \eqref{ex_b}. 

\subsubsection{Finite element formulation}
\label{S:FEM}
We introduce a triangulation $\mathcal{T}_h$, fitted to the boundary of $\Omega$.
The set of faces of triangles that form the boundary $\partial \Omega$
of $\Omega$ is denoted $\mathcal{F}$.

We will use the following notation for mesh related quantities. 
Let $h_K$ be the diameter of $K$ and $h = \max_{K \in \mathcal{T}_h} h_K$.
We introduce the finite element spaces
\[
V_h := \{v \in H^1(\Omega): v\vert_K \in P_1(\Omega),\,\forall K \in \mathcal{T}_h \}
\]
and
\[
\Lambda_h := \{\mu \in L^2(\partial \Omega): \mu \vert_F \in P_0(F),
\,\forall F \in \mathcal{F}\}.
\]
It is known that this choice of spaces does not satisfy
\eqref{brezzi_disc}.

The standard finite element formulation writes: find $u_h \in V_h$ and $\lambda_h \in \Lambda_h$ such that
\begin{equation}\label{standard_fem}
a(u_h,v_h) + b(\lambda_h,v_h) + b(\mu_h,u_h) = (f,v_h)+b(\mu_h,g)
 \quad \mbox{ for all } (v_h,\mu_h) \in V_h \times \Lambda_h.
\end{equation} 
Assume that $L_h$ denotes a coarsened version of $\Lambda_h$, $L_h \subset \Lambda_h$ such that
the inf-sup condition is uniformly satisfied for the pair $V_h \times
L_h$, we now that this is always possible if $L_h$ is chosen coarse
enough, i.e. if $H$ denotes the mesh size of $L_h$ thereholds $H> c
h$, for some $c>1$ and we assume that there exists a positive constant
$c_H$ such that $c_H H \leq h$. We let $\pi_L$ denote the $L^2$-projection on the
space $L_h$. As proposed in the previous section we may
stabilise the formulation \eqref{standard_fem} by adding the penalty term
\[
s(\lambda_h,\mu_h) = \left<h(\lambda_h -  \pi_{L} \lambda_h),
  \mu_h - \pi_L \mu_h\right>_{\partial \Omega}.
\]
Clearly the space $\Lambda_h$ is more convenient to work in since it
does not require any special meshing of the boundary.
If we now let $\mathcal{X}:=\{x_i\}$ be the set of all the mesh nodes in $\partial
\Omega$ excluding corner nodes.
Then there holds, by the arguments of Section \ref{projpenal}
\[
\|\lambda_h - \pi_L \lambda_h\|^2_{-\frac12,h,\partial \Omega} \leq c \sum_{x_j \in \mathcal{X}} h^2 \jump{\lambda_h}|^2_{x_j}.
\]
This prompts the stabilisation operator
\[
s(\lambda_h,\mu_h) :=  \sum_{x_j \in \mathcal{X}} h^2 \jump{\lambda_h}|_{x_j}\jump{\mu_h}|_{x_j}.
\]
Observe that penalising the jump of $\lambda_h$ over a corner node
leads to an inconsistent method even for smooth $u$, since $\lambda$
will jump across the corner due to the jump in the boundary normal.
The stabilised method the reads: find $u_h \in V_h$ and $\lambda_h \in \Lambda_h$ such that
\begin{multline}\label{stab2_fem}
a(u_h,v_h) + b(\lambda_h,v_h) + b(\mu_h,u_h)- \gamma s(\lambda_h,\mu_h) \\
= (f,v_h)+b(\mu_h,g) \quad \mbox{ for all } (v_h,\mu_h) \in V_h \times \Lambda_h.
\end{multline} 

We will outline the analysis of the penalty stabilised Lagrange
multiplier method using the abstract framework derived in Section
\ref{abstractset}
\subsubsection{Satisfaction of the assumptions of the abstract analysis}
We may now use the abstract analysis of Theorem \ref{best_approx_stab} combined with
Lemma \ref{stab_infsup} to prove a best approximation result. We will
use the discrete norms
\[
\|u_h\|_{V_h}= \|u_h\|_{1,h}, \quad \|\lambda_h\|_{L_h} := \|\lambda_h\|_{-\frac12,h,\partial \Omega}.
\]
By assumption $L_h$ satisfies the inf-sup condition
\eqref{brezzi_disc}, for $\pi_L$ defined as the $L^2$-projection on
the piecewise constants \eqref{stab_cond_pil} holds and hence we have the stabilised inf-sup condition \eqref{stab_infsup}.
It is easy to see that the continuities \eqref{acontdisc} and \eqref{bcontdisc} hold.
The condition \eqref{coercivity_assumption} also holds by taking
$\xi_h(u_h) := \delta h^{-1} \pi_L u_h$, where $\delta \in \mathbb{R}^+$.
The satisfaction of \eqref{coercivity_assumption} now follows from the
construction of $s(\cdot,\cdot)$, the quasi-uniformity between $H$ and
$h$, the stability of the $L^2$-projection and the definition of the
$L_h$ and $V_h$ norms,
\begin{equation}\label{coercivity_satisfaction}
c_0 s(\xi_h(u_h),\xi_h(u_h))^{\frac12} \leq\| h^{-1} \pi_L u_h\|_{L_h} \leq \|u_h\|_{\frac12,h,\partial \Omega} \leq
\|u_h\|_{V_h}.
\end{equation}
The second relation of \eqref{coercivity_assumption} is satisfied
using  the approximation property of the projection $\pi_L$ 
\[
\|u_h -
\pi_L u_h\|_{\frac12,h,\partial \Omega} \leq c_0 \|\nabla u_h \times n_{\partial \Omega}\|_{-\frac12,h,\partial \Omega}
\]
and a discrete trace inequality $\|\nabla u_h \times n_{\partial \Omega}\|_{-\frac12,h,\partial \Omega}
\leq c_t \|\nabla u_h\|_{L^2(\Omega)}$, leading to
\begin{equation}
\|u_h\|_{\frac12,h,\partial \Omega}^2 \leq 2 \|u_h -
\pi_L u_h\|_{\frac12,h,\partial \Omega}^2 + 2\|\pi_L u_h\|_{\frac12,h,\partial \Omega}^2 \leq
2\|\pi_L u_h\|_{\frac12,h,\partial \Omega}^2 + 2c^2_0 c^2_t \|\nabla u_h\|_{L^2(\Omega)}^2.
\end{equation}
It follows that 
\[
a(u_h,u_h) +  b(\xi_h(u_h),u_h) \ge  (1 - c^2_0 c^2_t \delta)  \|\nabla
u_h\|_{L^2(\Omega)}^2+ \frac{\delta}{2} \|u_h\|_{\frac12,h,\partial \Omega}^2
\]
and hence for $\delta <  c^{-2}_0 c^{-2}_t$ the coercivity assumption is
satisfied.

We conclude that the assumptions of Theorem \ref{best_approx_stab} are satisfied and
that the formulation \eqref{stab2_fem} is wellposed and satisfies a
best approximation result.
\subsection{Unfitted finite element methods and multi-model
  coupling}\label{unfitted}
Here we will consider the coupling of two models of elasticity over a smooth
interface that is not fitted to the computational mesh. This type of
method
can be useful for problems where the interface itself is an unknown
and repeated computations have to be performed with different
interface positions, for instance for transient problems where an interface moves through
the mesh or for inverse identification where the interface will move
during iterations.

We consider a geometrical setting where a polygonal $\Omega$ is
decomposed in two
subdomains, $\Omega_1$ and $\Omega_2$ and a separating interface
$\Gamma$. In each subdomain $\Omega_i$ we consider the following 
partial differential equation:
\[
\nabla \cdot  \sigma_i(u_i) = f,\quad \mbox{ in } \Omega_i
\]
where $u_i \in V_i := [H^1(\Omega_i)]^2$ denotes a displacement field,
$\sigma_i(u_i) \in [H(\mbox{div};\Omega_i)]^2$ the stress tensor and
$f \in L^2(\Omega)$ the applied force. Across
the interface we assume that the following matching conditions hold
\[
u_1 - u_2 =0,\quad (\sigma_1(u_1) - \sigma_2(u_2)) \cdot n_\Gamma = 0.
\] 
For simplicity we assume that $u=0$ on the outer boundary $\partial
\Omega$. Let $$V:= \{(v_1,v_2) \in  V_1 \times V_2 : v_i\vert_{\partial
  \Omega_i \cap \partial \Omega} = 0 \}$$ and $L$ be the dual space to
the space of traces of $V$ on $\Gamma$.
We propose the following norm on $V$:
\begin{equation}\label{multi_phys_norm}
\|u\|_V := \sum_{i=1}^2 \|\nabla u\|_{\Omega_i} + \|u_1 - u_2\|_{\frac12,\Gamma}.
\end{equation}
We assume that the following coercivity and continuity properties hold for the continuous
problem. There exists positive constants $\alpha_0,\, \alpha_1, M$
such that
\begin{equation}\label{elast_korn}
\alpha_0 \sum_i  \|\nabla u_i\|_{\Omega_i} ^2 \leq \sum_{i=1}^2
\left((\sigma_i(u_i),\nabla u_i)_{\Omega_i} + \|u_i\|_{\partial
    \Omega_i \cap \partial \Omega}^2\right) + \|u_1 - u_2\|_\Gamma^2,\quad
\forall (u_1,u_2) \in V_1 \times V_2, 
\end{equation}
\begin{equation}\label{elast_coerce}
\alpha_1 \|u\|_V^2 \leq \sum_{i=1}^2 (\sigma_i(u_i),\nabla u_i)_{\Omega_i},\quad
\forall u \in \{v \in V: \left<\nu,v_1 - v_2\right>_\Gamma
=0,\,\forall \nu \in L\}, 
\end{equation}
\begin{equation}\label{elast_cont}
|\sum_{i=1}^2 (\sigma_i(u_i),\nabla
v_i)_{\Omega_i}| \leq M \|u\|_V \|v\|_V ,\quad \forall u,v \in V.
\end{equation}
Note that \eqref{elast_korn} typically implies a Korn's inequality and
that \eqref{elast_coerce} is a consequence of \eqref{elast_korn}, the
boundary and interface conditions and the Poincar\'e inequality.
We propose a weak formulation using
Lagrange multipliers that takes the form, find $(u,\lambda) \in
V\times L$ such that
\begin{equation}\label{weak_multiphys}
a(u,v) + b(\lambda,v) + b(\mu,u) = (f,v)
 \quad \mbox{ for all } (v,\mu) \in V \times L,
\end{equation} 
where
\begin{equation}\label{bilinear_mulitphys}
a(u,v) := \sum_{i=1}^2 (\sigma_i(u_{i}),\nabla v_i)_{\Omega_i},\quad 
b(\lambda,v) = \left<\lambda,(v_{1} - v_{2})\right>_{\Gamma}.
\end{equation}
Note that the continuity $b(\lambda,v) \leq M_b \|\lambda\|_L \|v\|_V$
holds.
If in addition to \eqref{elast_korn}, \eqref{elast_coerce} and the
above continuities we assume that $\sigma_i(u_i)$ are linear, this formulation is wellposed by the Babuska-Brezzi Theorem (see
\cite{Ba72,Bre74}). Observe that there are some differences in the
functional analytical framework depending on whether or not $\Gamma$
intersects the Dirichlet boundary. These differences are irrelevant
for the present discussion and will be neglected.
\subsubsection{Finite element formulation}
Consider the mesh family $\{\mathcal{T}_h\}_h$ where we let $\mathcal{T}_h:= \{K\}$ be a triangulation of $\Omega$ that is
constructed
without fitting the element nodes or sides to the interface
$\Gamma$. For any $\mathcal{T}_h$ we now extract two subtriangulations,  $\mathcal{T}_i:= \{K
\in \mathcal{T}_h: K \cap \Omega_i \ne \emptyset\}$, i=1,2. We
define two finite element spaces, one for $\Omega_1$ and one for
$\Omega_2$ by
\[
V_{ih} := \{v \in V_i: v\vert_K \in [P_1(K)]^2,\, \forall K
\in \mathcal{T}_i \mbox{
  and } v\vert_{\partial \Omega \cap \mathcal{T}_i} = 0\}.
\]
 Let $\tilde G_h:= \{K \in
 \mathcal{T}_h: K \cap \Gamma \ne \emptyset \}$. We assume that the
 mesh is fine enough so that, for all $K \in \tilde G_h$, $\Gamma \cap K$ can be approximated by a
 line segment, i.e. that $\Gamma$ intersects the boundary of $K$ in
two points and that there exists $c>0$ so that $\mbox{meas}(\Gamma
 \cap K) < ch$ for all elements and all meshes.

Observe that the finite element functions extend to all of the mesh
domain $\mathcal{T}_i$ which can lead to conditioning problems if
there are elements in $\tilde G_h$ with very small intersection with
the physical domain. On the set $\tilde G_h$ we define the following multiplier space
\[
\Lambda_h := \{\lambda_h \in [L^2(\tilde G_h)]^2: \lambda_h\vert_K \in
[P_0(K)]^2,\, \forall K \in \tilde G_h \}.
\]
The Lagrange multiplier is defined on the same elements
as the primal variables and hence has been extended in space, the
advantage of this is that the stabilisation of the multiplier can
be designed on the standard volume elements (here in $\mathcal{R}^2$) and we do not need to
consider a trace mesh of $\Gamma$.

The finite element method once again is on the generic form
find $u_h:= \{u_{h1},u_{h2} \} \in V_{1h} \times V_{2h} =: V_h$ and $\lambda_h \in \Lambda_h$ such that
\begin{multline}\label{stab_fem_gen}
a(u_h,v_h) + b(\lambda_h,v_h) + b(\mu_h,u_h) - s(\lambda_h,\mu_h) = (f,v_h)\\
 \quad \mbox{ for all } (v_h,\mu_h) \in V_h \times \Lambda_h,
\end{multline} 
where the bilinear forms $a(\cdot,\cdot)$ and $b(\cdot,\cdot)$ are
defined by \eqref{bilinear_mulitphys} and $s(\cdot,\cdot)$ will be
detailed below.
We know that if we instead looked for $\lambda_h$ in a space $L_h$ defined on macro
elements with diameter $H$ such that $H> c_h h$, with $c_h$
sufficiently large the inf-sup condition would be satisfied. We also assume that
there exists $c_H>0$ so that $c_H H \leq h$. We assume that the
space $L_h$ is constructed by assembling elements in $\tilde G_h$ into
macro patches $F_j$ such that for every $j$ $H \leq \mbox{meas}(F_j \cap
\Gamma) \leq H+h$. By the constraints on the mesh with respect to the
interface we may conclude that the cardinality of the set $\{K: K\cap
F_j \ne \emptyset\}$ is upper
bounded uniformly in $j$ and $h$  by some $M_F \in \mathbb{N}^+$. To
each boundary patch $F_j$ we associate a shape regular macro patch
$\omega^i_j \subset \Omega_i$ consisting $F_j \cap \Omega_i$ and a sufficient number of interior
elements $K \subset \mathcal{T}_{ih} \cap \Omega_i$ so that $\mbox{meas}(\omega^i_j \cap \Omega_i)
= O(H^2)$. It follows by construction that $\bar \omega^1_j \cap \bar \omega^2_j
= F_j$ and we assume that for fixed $i$, the interiors of the patches $\omega^i_j$ are
disjoint. The rationale for the patches $\omega^i_j$ is that for all
$u_j \in H^1(\omega^i_j)$ the following trace inequality holds
\begin{equation}\label{trace_multiphys}
H^{-\frac12} \|u_j - \pi_L u_j\|_{\Gamma \cap \omega^i_j} \leq c_P
\|\nabla u_j\|_{\omega^i_j}
\end{equation}
where $\pi_L$ denotes the projection onto piecewise constant functions
on $F_j$ and $c_P$ is independent of the mesh interface intersection. This
inequality is proven by mapping to a reference patch $\hat \omega$,
there applying a trace inequality followed by a Poincar\'e type
inequality (see Corollary B.65 of
\cite{EG04}) and then mapping back to the
physical patch $\omega^i_j$, using the shape regularity of the patch
for uniformity.  For completeness we sketch a proof of the construction of the Fortin interpolant in
appendix. Observe that using the
stable pair $V_h \times L_h$ and taking $s(\cdot,\cdot)=0$ then leads to a
best approximation for the inf-sup stable unfitted finite element
method using Theorem \ref{best_approx_stab}.

As before we get the abstract stabilisation operator
\begin{equation}\label{abst_stab}
s(\lambda_h,\mu_h) := \left<h (\lambda_h - \pi_L \lambda_h), \mu_h - \pi_L \mu_h\right>_\Gamma.
\end{equation}
In practice, since we do not want to be concerned with the construction
of $L_h$ we apply the ideas of section \ref{projpenal} and instead work with the operator
\begin{equation}\label{pract_stab}
s(\lambda_h,\mu_h) := \sum_{K\in \tilde G_h} \int_{\partial K
  \setminus \partial \tilde  G_h} h \jump{\lambda_h}\cdot \jump{\mu_h} ~\mbox{d}s,
\end{equation}
where $\jump{x}$ denotes the jump of the quantity $x$ over the
interior faces of the elements in the set $\tilde G_h$.
\begin{remark}
Note that although the operator of \eqref{abst_stab} is defined on
$\Gamma$ the operator \eqref{pract_stab} is defined on the interior
faces of elements in $G_h$. This convenient trick introduced in
\cite{BH10b}, allows us to use the volume mesh structure for stabilisation and we
never need to worry about the actual intersections of $\Gamma$ with
element boundaries. Uniformity of the stabilisation relies on the mesh
regularity.
\end{remark}
\subsubsection{Satisfaction of the assumptions of the abstract
  analysis}
For the method to be robust with respect to the mesh-interface
intersection the constants in the bounds in the above abstract
analysis must all be independent of the cut. This holds
for the approximation using the inf-sup stable space $V_h \times L_h$,
thanks to the robustness of the Fortin interpolant and the properties
of $a(\cdot,\cdot)$ and $b(\cdot,\cdot)$.
Therefore we only show that the inequalities
\eqref{coercivity_assumption}
also can be made independent of the cut, under the above assumptions.
Similarly as in the case of weak boundary condition we introduce the
following norms on the discrete spaces
$$\|\lambda_h\|^2_{\frac12,h,\Gamma} = \left<h^{-1}\lambda_h,\lambda_h\right>_{\Gamma}
,\quad \|\lambda_h\|^2_{-\frac12,h,\Gamma} =
\left<h\lambda_h,\lambda_h\right>_{\Gamma},$$
$$
\|u_h\|^2_{1,h}:=\sum_{i=1}^2 \|\nabla u_{ih}\|^2_{L^2(\Omega_i)} +
 \|u_{1h}-u_{2h}\|_{\frac12,h,\Gamma}^2.
$$
To prove that the hypothesis of Theorem \ref{best_approx_stab}
are satisfied we chose the norms $\|\cdot \|_{V_h}$ and $\|\cdot
\|_{L_h}$ as follows,
\[
\|u \|_{V_h}:=\|u\|_{1,h},\quad \|\lambda\|_{L_h}:=\|\lambda\|_{-\frac12,h,\Gamma}.
\]
To satisfy the coercivity condition of \eqref{coercivity_assumption}
we take
$\xi_h(u_h)\vert_{F_i} := \delta H^{-1} \pi_L (u_1-u_2)$. 
We recall that $\pi_L$ is defined by the
projection on the space $L_h$ with mesh size $H$, $$
\pi_L(u_{1h}-u_{2h})\vert_{F_i} := |F_i\cap \Gamma|^{-1} \int_{F_i \cap \Gamma}  (u_{1h}-u_{2h})
~\mbox{d}s.$$
By this choice,
using the orthogonality of $\pi_L$ we have
\begin{multline*}
 b(\xi_h(u_h),u_h) = \delta \sum_i \left< H^{-1} \pi_L (u_{1h}-u_{2h}),
   \pi_L (u_{1h}-u_{2h}) \right>_{F_i \cap \Gamma}\\
=  \delta \sum_j \|H^{-\frac12} (\pi_L - I) (u_{1h}-u_{2h}),
  \|^2_{F_j \cap \Gamma} + \delta \sum_j \| H^{-\frac12} (u_{1h}-u_{2h}),
 \|^2_{F_j \cap \Gamma}\\
\ge -2 \delta \sum_{i=1}^2 \sum_j \| H^{-\frac12}  (\pi_L - I)  u_{ih}\|^2_{\omega^i_j \cap \Gamma} + \delta\sum_j \| H^{-\frac12} (u_{1h}-u_{2h})
 \|^2_{F_j \cap \Gamma}.
\end{multline*}
Applying \eqref{trace_multiphys} in the right hand side of the last
inequality It then follows that 
\[
 a(u_h,u_h) +  b(\xi_h(u_h),u_h) \ge \sum_{i=1}^2
 (\sigma_i(u_{ih}),\nabla u_{ih})_{\Omega_i} - 2 \delta c_P^2 \sum_{i=1}^2
 \|\nabla u_{ih}\|^2_{\Omega_i} + 
 \delta c_H \|u_{1h}-u_{2h}\|_{\frac12,h,\Gamma}^2.
\]
We then apply \eqref{elast_korn} in the right hand side, recalling
that $u_{ih} \in V_i$, $i=1,2$ to obtain
\[
 a(u_h,u_h) +  b(\xi_h(u_h),u_h) \ge (\alpha_0- 2 \delta c_P^2 ) \sum_{i=1}^2
 \|\nabla u_{ih}\|^2_{\Omega_i} + (\delta c_H - h) \|u_{1h}-u_{2h}\|_{\frac12,h,\Gamma}^2
\]
and we conclude by choosing $\delta = \tfrac{\alpha_0}{4 c_P^2}$ and
taking $h < \delta c_H$.

For the second inequality of \eqref{coercivity_assumption} observe
that by the fact that an interface segment $F_j$ can only
be cut by a uniformly upper bounded number of elements, the mesh
condition, $c_h h \leq H
\leq c_H^{-1} h$, and that the $\xi_h(u_h)$ are
constant over each macro patch $F_j$ we have
\begin{multline*}
\sum_{K\in \tilde G_h} \int_{\partial K
  \setminus \partial \tilde  G_h} h \jump{\xi_h(u_h)}^2~\mbox{d}x \lesssim
\sum_{K\in \tilde G_h} h^2 |\xi_h(u_h)\vert_{K}|^2\\ \lesssim M_F \sum_j h^2
|\xi_h(u_h)\vert_{F_j}|^2 \lesssim \|\xi_h(u_h)\|_{L_h}^2.
\end{multline*}
Then using the stability of the $L^2$-projection and the mesh
conditions linking $h$ and $H$ we conclude
\[
\|\xi_h(u_h)\|_{L_h} = \delta \|H^{-1} \pi_L (u_{1h}-u_{2h})\|_{-\frac12,h,\Gamma}
\lesssim \|\pi_L (u_{1h}-u_{2h})\|_{\frac12,h,\Gamma} \lesssim \|u_h\|_{V_h}.
\]

We conclude that the results of Theorem \ref{best_approx_stab} hold in
this case as well.
\begin{remark}
 By using
suitable extensions of the solution following 
\cite{BH10a} and \cite{HH04} optimal convergence may be obtained for
smooth solutions. The conditioning of the system however
depends on how the interface cuts the mesh and must be handled either
following the ideas introduced in \cite{Bu11b} or by preconditioning.
\end{remark}
\subsection{Nitsche's method and stabilised Lagrange multiplier methods: a
different approach}
The close relation between the residual based stabilised methods for
Lagrange multipliers as introduced by Barbosa and Hughes and Nitsche's
method
was discussed by Stenberg in \cite{Sten95}. The idea of that paper was that
if the Lagrange multiplier can be eliminated locally by solving the
constraint equation, Nitsche's method is recovered. Other authors have recently
discussed the need of penalty for Nitsche's method and its close
relation to Lagrange multiplier methods, see for instance
\cite{Bu11a,DBDV10, Gro10}.

Herein we will show the
connection between the non-symmetric variant of Nitsche's method, the projection stabilised methods discussed above
and the residual based stabilisation of the Lagrange multiplier. Let
us first recall the nonsymmetric version of the method of Barbosa \& Hughes: find $\{u_h,\lambda_h\}
\in V_h \times \Lambda_h$ such that
\begin{multline}\label{BB_stab}
A_{BH}[(u_h,\lambda_h),(v_h,\mu_h)]\\
:=a(u_h,v_h)+b(\lambda_h,v_h)-b(u_h,\mu_h)+\gamma \left<h(\lambda_h + \nabla u_h \cdot n),
\mu_h + \nabla v_h \cdot n\right>_{\partial \Omega} \\ = (f,v_h), \, \forall \{v_h,\mu_h\} \in V_h \times \Lambda_h,
\end{multline}
with $a(\cdot,\cdot)$ and $b(\cdot,\cdot)$ are defined by \eqref{ex_a}
and \eqref{ex_b}, corresponding to the weak imposition of boundary conditions.
Recalling that formally the Lagrange multiplier is given by the
diffusive flux $\lambda = -\nabla u \cdot n$, we immediately conclude
that the method is consistent. Stability is then typically proven by
testing with $v_h=u_h$ and $\mu_h = \lambda_h$ using the positivity of
the form to obtain control of $\|h^\frac12 \lambda_h\|_{\partial
  \Omega}$ by absorbing all the other terms in the stabilisation using
the $H^1$-seminorm of $u_h$ over the domain. Control of $u_h$ on the
boundary is then obtained in a second step by choosing $\mu_h$ suitably.

We will now consider the stabilisation used in \eqref{BB_stab} as a
penalty on the distance to a stable subspace. This would mean using the space $N_h$ of normal derivatives of $V_h$ on the trace
mesh as multiplier space, together with $V_h$ for the primal variabel. Since in that
case $V_h$ and $N_h$ no longer can be chosen
independently this method may be written: find $u_h \in V_h$ such that
\begin{multline}\label{non_sym_Nitsche}
A_{Nit}(u_h,v_h)
:=a(u_h,v_h)+b(-\nabla u_h\cdot n,v_h)-b(u_h,-\nabla v_h \cdot n)\\=
(f,v_h), \, \forall v_h \in V_h.
\end{multline}
We have eliminated the Lagrange multiplier in the formulation using
its equivalence with the diffusive flux. Writing out this variational
formulation leads to
\[
\int_\Omega \nabla u_h \cdot \nabla v_h ~\mbox{d} x - \int_{\partial
  \Omega} \nabla u_h \cdot n v_h  ~\mbox{d} s + \int_{\partial
  \Omega} \nabla v_h \cdot n u_h ~\mbox{d} s = \int_\Omega f v ~ \mbox{d}x,
\]
which we identify as the non-symmetric version of Nitsche's method,
without penalty. For the argument to make sense we now need a
stability result for this method. The question of the inf-sup stability of the non-symmetric version of Nitsche's method,
without penalty,
was recently treated in \cite{Bu11a}, where the following stability
result was proven:
\begin{lemma}\label{stab_non_sym_Nit}
Let $V_h$ be the standard space of piecewise polynomial continuous
finite element functions.
Assume that the each face of the polygonal $\Omega$ is mesh with a sufficient
number of elements (depending only on the shape regularity), then for some $\zeta \ge c_0 >0 $, with $c_0$
independent of $h$, but not of the mesh geometry, there holds
\[
\|u_h\|_{1,h} \lesssim \sup_{v_h \in V_h} \frac{A_{Nit}(u_h,v_h)}{\|v_h\|_{1,h}},
\]
where 
\[
\|u_h\|_{1,h}^2 := \|\nabla u_h\|_{L^2(\Omega)}^2 + \zeta \|h^{-\frac12}
u_h\|_{L^2(\partial \Omega)}^2,\quad \zeta>0.
\]
\end{lemma}
It follows that we have the required stability and we may prove
stability of the residual based stabilisation using the techniques
discussed above. 
\begin{remark}\label{stab_func}
The above lemma can be rewritten as $\exists w_h \in
V_h$ such that
\begin{equation}\label{stab_func1}
c_w \|u_h\|^2_{1,h} \leq A_{Nit}(u_h,w_h)
\end{equation}
with $w_h:= u_h + \zeta \varphi_{\partial}$,  $c_w>0$ and
\begin{equation}\label{stab_func2}
\|\varphi_{\partial}\|_{1,h} \leq c_{\partial} \|u_h\|_{\frac12,h,\partial \Omega}.
\end{equation}
The function $\varphi_{\partial}$ ensures the control of the boundary contribution.
\end{remark}
We now give an alternative proof of the equivalent of Lemma
\ref{stab_non_sym_Nit}
for the formulation \eqref{BB_stab} using the framework of penalty on
the distance to the stable subspace. The result holds for 
multiplier spaces satisfying the following compatibility assumption.\\[5mm]
{\bf Assumption [A1]:} the following continuity holds for the spaces $V_h$ and
$\Lambda_h$.
For every $v_h \in V_h$ there exists $z_h(v_h) \in \Lambda_h$ such
that 
\begin{equation}\label{compatible}
b(u_h, \nabla v_h\cdot n + z_h(v_h)) \leq C_z \|\nabla
u_h\|_{L^2(\Omega)} \|\nabla v_h\|_{L^2(\Omega)}, \quad \|h^{\frac12}
z_h(v_h)\|_{L^2(\partial \Omega)} \leq c_z \|\nabla v_h\|_{L^2(\Omega)}.
\end{equation}

\begin{theorem}\label{BBNitstab}
Let $V_h\times\Lambda_h$ satisfy assumption {\bf [A1]}.
Then for all $\{u_h,\lambda_h\} \in V_h\times \Lambda_h$ there holds 
\begin{equation}
\|u_h\|_{1,h} + \|\lambda_h\|_{L_h} \lesssim \sup_{\{w_h,\nu_h\} \in V_h
  \times \Lambda_h}\frac{A_{BH}[(u_h,\lambda_h),(w_h,\nu_h)]}{\|w_h\|_{1,h} + \|\nu_h\|_{L_h} }.
\end{equation}
\end{theorem}
\begin{proof}
First take $v_h = u_h$ and $\mu_h = \lambda_h$ to obtain
\begin{multline}
\|\nabla u_h\|_{L^2(\Omega)}^2 + \gamma \|h^{\frac12} (\lambda_h +
\nabla u_h)\|_{L^2(\partial \Omega)}^2 \\
= a(u_h,u_h)+b(\lambda_h,u_h)-b(\lambda_h, u_h)+\gamma \left<h(\lambda_h + \nabla u_h \cdot n),
\lambda_h + \nabla u_h \cdot n\right>_{\partial \Omega}
\end{multline}
We add and subtract $\nabla u_h \cdot n$ and $\nabla v_h \cdot n$ in
the $b(\cdot,\cdot)$ forms of the formulation
\begin{multline}\label{rewrite}
a(u_h,v_h)+b(\lambda_h,v_h)-b(u_h,\mu_h)+\gamma \left<h(\lambda_h + \nabla u_h \cdot n),
\mu_h + \nabla v_h \cdot n\right>_{\partial \Omega} \\ 
= a(u_h,v_h)+b(-\nabla u_h \cdot n,v_h) +b(\nabla u_h \cdot
n+\lambda_h,v_h)\\
-b(u_h,-\nabla v_h \cdot
n)-b(u_h,\nabla v_h \cdot
n+\mu_h)+\gamma \left<h(\lambda + \nabla u_h \cdot n),
\mu_h + \nabla v_h \cdot n\right>_{\partial \Omega}\\
= A_{Nit}(u_h,v_h)\\
+b(\nabla u_h \cdot n+\lambda_h,v_h) -b(u_h,\nabla v_h \cdot
n+\mu_h) +\gamma \left< h(\lambda_h + \nabla u_h \cdot n),
\mu_h+\nabla v_h \cdot n \right>_{\partial \Omega}.
\end{multline}
We will first show that by taking $v_h=w_h$ (of \eqref{stab_func1}-\eqref{stab_func2}) and
$\mu_h := \lambda_h + z_h(\zeta \varphi_\partial)$ we have
\begin{equation*}
\|u_h\|^2_{1,h} + \gamma h \|\lambda_h + \nabla u_h \cdot
n\|_{L^2(\partial \Omega)}^2 \lesssim A_{BH}[(u_h,\lambda_h),(w_h,\lambda_h + z_h(\zeta \varphi_\partial))].
\end{equation*}
First note that by the construction of $w_h$ and $\mu_h$ and the form
\eqref{rewrite} we have
\begin{multline}\label{BBtoNit}
A_{BH}[(u_h,\lambda_h),(w_h,\mu_h)] = A_{Nit}(u_h,w_h) + \gamma \|h^{\frac12} (\lambda_h +
\nabla u_h\cdot n)\|_{L^2(\partial \Omega)}^2  \\
+\zeta b(\nabla u_h \cdot n+\lambda_h,\varphi_\partial) -\zeta b(u_h,\nabla \varphi_\partial \cdot
n+z_h(\varphi_\partial)) \\ +\gamma \zeta \left<h(\nabla u_h \cdot n+\lambda_h ),
\nabla \varphi_\partial \cdot
n+z_h( \varphi_\partial)\right>_{\partial \Omega}.
\end{multline}
Then note that by the continuity of $b(\cdot,\cdot)$ and by using the
Cauchy-Schwarz inequality in the penalty term we have
\begin{multline}\label{bound1}
\zeta b(\nabla u_h \cdot n+\lambda_h,\varphi_\partial) \leq c_b \|h^{\frac12} (\lambda_h +
\nabla u_h\cdot n)\|_{L^2(\partial \Omega)} \zeta c_\partial
\|u_h\|_{\frac12,h,\partial \Omega} \\
\leq \frac14 \gamma \|h^{\frac12} (\lambda_h +
\nabla u_h\cdot n)\|_{L^2(\partial \Omega)}^2 + \gamma^{-1} \zeta^2 c_b^2
c_{\partial}^2 \|u_h\|_{\frac12,h,\partial \Omega}^2,
\end{multline}
\begin{multline}\label{bound2}
\zeta b(u_h,\nabla \varphi_\partial \cdot
n+z_h(\varphi_\partial)) \leq \zeta C_z \|\nabla
u_h\|_{L^2(\Omega)} \|\varphi_\partial\|_{1,h} \\
\leq \frac12 c_w \|\nabla
u_h\|_{L^2(\Omega)}^2 + C_z^2 c^{-1}_w c_\partial^2 \zeta^2 \|u_h\|_{\frac12,h,\partial \Omega}^2
\end{multline}
and 
\begin{multline}\label{bound3}
\gamma \zeta (h(\nabla u_h \cdot n+\lambda_h ),
\nabla \varphi_\partial \cdot
n+z_h( \varphi_\partial)) \lesssim \frac14 \gamma \|h^{\frac12} (\lambda_h +
\nabla u_h\cdot n)\|_{L^2(\partial \Omega)}^2\\
+ 2 (c_t^2 c_{\partial}^2 +c_z^2) \zeta^2 \gamma \|u_h\|_{\frac12,h,\partial \Omega}^2.
\end{multline}
If we choose $\zeta$ small enough, it
follows from \eqref{BBtoNit}, \eqref{stab_func1}-\eqref{stab_func2} and the bounds
\eqref{bound1}-\eqref{bound3} that
\[
A_{BH}[(u_h,\lambda_h),(w_h,\mu_h)] \ge \frac12 c_w \|u_h\|_{1,h}^2 +
\frac12 \gamma \|h^{\frac12} (\lambda_h +
\nabla u_h\cdot n)\|_{L^2(\partial \Omega)}^2.
\]
Since $\|h^\frac12 \lambda_h\|_{L^2(\partial \Omega)} \leq \|h^{\frac12} (\lambda_h +
\nabla u_h\cdot n)\|_{L^2(\partial \Omega)} + c_t \|\nabla
u_h\|_{L^2(\Omega)}$ we deduce that 
\[
\|u_h\|_{1,h}^2 + \|h^\frac12 \lambda_h\|_{L^2(\partial \Omega)}^2
\lesssim A_{BH}[(u_h,\lambda_h),(w_h,\mu_h)].
\]
It only remains to show that 
\[
\|u_h+\zeta \varphi_\partial\|_{1,h} + \|h^\frac12 (\lambda_h+z_h(\zeta \varphi_\partial))\|_{L^2(\partial \Omega)} \lesssim \|u_h\|_{1,h} + \|h^\frac12 \lambda_h\|_{L^2(\partial \Omega)}.
\]
This is immediate by the triangle inequality and the stability
\[
\| \varphi_\partial\|_{1,h}  +\|h^{\frac12} z_h(\zeta
\varphi_\partial)\|_{L^2(\partial \Omega)} \lesssim \|u_h\|_{\frac12,h,\partial \Omega}.
\]
\end{proof}
\begin{remark}
The condition \eqref{compatible} is easily satisfied for any
reasonable space
$\Lambda_h$. For spaces including discontinuous functions on boundary elements $F_j$ take
$z_h(v_h)\vert_{F_j} := -\mbox{meas}(F_j)^{-1} \int_{F_j} \nabla v_h \cdot n ~\mbox{d}s$. If
the spaces $\Lambda_h$ consists of continuous functions decompose the
boundary in macro patches $F_j$ consisting of a sufficient number of elements for the construction
of functions $z_h(v_h) \in H^1_0(F_j)$ such that $\int_{F_j}
z_h(v_h)~\mbox{d}s = -\int_{F_j} \nabla v_h \cdot n ~\mbox{d}s$. Then
on each subdomain $F_j$ there holds (with $\pi_{L}$ denoting the
$L^2$-projection on constant functions on $F_j$)
\[
\left<u_h, \nabla v_h \cdot n + z_h(v_h)  \right>_{F_j} = \left<u_h - \pi_{L} u_h, \nabla v_h \cdot n + z_h(v_h)  \right>_{F_j} . 
\]
It also follows that whenever the choice $z_h(v_h) = - \nabla v_h
\cdot n$ is possible, the right hand sides of \eqref{bound2} and
\eqref{bound3} are zero and therefore the stability is obtained
independently of the stability parameter $\gamma$. It is then
straightforward to show, using the above inf-sup argument, that the solution $u_h$ of
\eqref{BB_stab} converges to that of \eqref{non_sym_Nitsche} in the
limit $\gamma \rightarrow \infty$. This is consistent with the
argument of \cite{Sten95}, since the local elimination of the Lagrange
multiplier in \eqref{BB_stab} yields the non-symmetric version of
Nitsche's method with a penalty that vanishes in the limit $\gamma \rightarrow \infty$.
\end{remark}
\begin{remark}
It follows from Theorem \ref{BBNitstab} that the nonsymmetric version
of the stabilisation of Barbosa and
Hughes, can be interpreted as a penalty on the distance to the stable
subspace,
consisting of the normal derivatives of the primal
finite element space in the setting of the non-symmetric Nitsche
method. Loosely speaking, we can
consider the non-symmetric Nitsche
method as a special member of the set of inf-sup stable Lagrange
multiplier methods. An associated stabilisation based on penalty on
the distance to a stable subspace is the Barbosa-Hughes method.
In case the Lagrange multiplier can be eliminated locally the two
methods are equivalent and the stabilised method is robust for large
values of the penalty parameter.
\end{remark}
\section{Numerical example}
The aim of this section is to compare the performance of the different
methods
in the simple case of weak imposition of boundary conditions. All
computations were carried out using Freefem++ \cite{freefem}.

We consider the Poisson problem in the unit square, $\Omega := (0,1)\times(0,1)$. The source term
and boundary terms are chosen so that
\[
u(x,y) = \frac{1}{2 \pi^2} \cos(\pi x) \cos(\pi y)+ 0.25 x (1-x) y (1-y).
\]
We compute the solution using the non-symmetric residual stabilised
method, a projection stabilisation method and penalty free Nitsche
methods. Below the forms $a(\cdot,\cdot)$ and $b(\cdot,\cdot)$ are
given by \eqref{ex_a} and \eqref{ex_b}. We impose Dirichlet boundary
conditions on the boundaries $y=0$ and $y=1$ (denoted $\partial
\Omega_D$ below). On the other two
boundaries we impose Neumann conditions. In all cases the primal
variable $u_h$ is approximated using continuous finite elements, of
first or second polynomial order,
\[
V_h^k := \{v_h \in C^0(\bar \Omega) : v_h \vert_K \in
\mathbb{P}_k(K),\, \forall K \in \mathcal{T}_h \},\, k=1,2.
\]
Let $G_h:=\{F\}$ denote a trace mesh on $\partial \Omega_D$,
coinciding with the trace mesh of $\mathcal{T}_h$ and $G_{\tilde h}:=\{F\}$ a trace mesh on $\partial \Omega_D$,
such that the local mesh size in $G_{\tilde h}$ is half that of
$\mathcal{T}_h$, $h=2 \tilde h$.
Define the Lagrange multiplier spaces by 
\[
\Lambda_h^1:=\{v_h \in L^2(\partial \Omega_D): v_h \vert_{F} \in
\mathbb{P}_0,\, \forall F \in G_{\tilde h}\},
\] 
\[
\Lambda_h^2:=\{v_h \in L^2(\partial \Omega_D): v_h \vert_{F} \in
\mathbb{P}_2,\, \forall F \in G_h\}.
\] 
These spaces are chosen so that the pair $V_h^k \times \Lambda_h^k$, $k=1,2$ are
unstable. The stabilising spaces were then both chosen as
\[
L_h^k:= \{v_h \in C^0(\partial \bar \Omega_D): v_h \vert_{F} \in
\mathbb{P}_1,\, \forall F \in G_h\},\, k=1,2.
\]
Here $C^0(\partial \bar \Omega_D)$ stands for functions continuous on
each separate connected component of $\partial \Omega_D$. It is
straightforward to verify that the spaces $V_h^k \times L_h^k$ are
stable for our problem. In all figures below square markers refer to
methods using $k=1$ and circles to methods using $k=2$. Empty markers
indicate convergence of the error in the $H^1$-norm and filled markers in
the $L^2$-norm.
We have also plotted for reference the slopes corresponding to
$O(h^\alpha)$ convergence with $\alpha=1$ in dotted line, $\alpha=2$
in a dashed line and $\alpha=3$ in dash dotted line. These reference
plots are the same for all methods so that the relative performance
can be assessed. In all cases the stabilisation parameter has been set
to $\gamma=1$. This parameter appeared to give a resonable result for
all methods. We observed that increasing the parameter can improve the
accuracy in the multiplier at the expense of the primal variable and
vice versa.

We consider the formulation \eqref{stab_fem} with the stabilisation
given by $$s(\lambda_h,\mu_h) = \left<\gamma h (\lambda_h - \pi_L \lambda_h,
  \mu_h - \pi_L \mu_h \right>_{\partial \Omega_D}$$ and the finite
element spaces proposed above. In figure \ref{fig:proj_stab}, left plot, we give
the convergence plots for $k=1$ and $k=2$. Then we consider the method
\eqref{BB_stab} and give the same convergence curves in the right plot
of \eqref{fig:proj_stab}. For comparison we also present the results obtained
using the inf-sup stable finite element pairs $V_h^k \times
L_h^k$. Finally we consider the penalty free version of Nitsche's
method, both the non-symmetric version given by equation
\eqref{non_sym_Nitsche} and its symmetric equivalent that may be written
\[
\int_\Omega \nabla u_h \cdot \nabla v_h ~\mbox{d} x - \int_{\partial
  \Omega} \nabla u_h \cdot n v_h  ~\mbox{d} s - \int_{\partial
  \Omega} \nabla v_h \cdot n u_h ~\mbox{d} s = \int_\Omega f v ~ \mbox{d}x - \int_{\partial
  \Omega} \nabla v_h \cdot n g~\mbox{d} s.
\]
Observe that the stability properties of this latter method are
unknown, but for the computations considered herein the method
remained stable and optimally convergent. We report the convergence of
the Nitsche type methods in the left plot of figure
\eqref{fig:Nitsche}. The symmetric version is distinguished by thick
lines. A consequence of the close relation between the Barbosa-Hughes
method and Nitsche's method is that in both the symmetric and the
non-symmetric case the unpenalised Nitsche methods are recovered in
the limit as the stabilisation goes to infinity. This is illustrated
in the right plot of \eqref{fig:Nitsche} where we show the variation
$L^2$-error of the difference between the 
solution obtained by the Barbosa-Hughes stabilisation and Nitsche's
method
on a $20\times20$ mesh as the penalty parameter goes to infinity. In
both the non-symmetric and the symmetric case the penalty free Nitsche
type methods are recovered. Observe the strong increase in the error
for the symmetric case at approximately $2.1$ where the matrix becomes
singular. For higher values of the penalty parameter no instabilities
were observed.

We make the following observations. The $H^1$-norm error is almost
identical for all methods. For the $L^2$-norm error all adjoint
consistent methods have very similar error curves, whereas the 
lack of adjoint consistency is expressed only as a larger error
constant 
and not in a loss of convergence order as expected from the
analysis. In experiments not reported here we imposed Dirichlet
conditions all around the domain to see the effects on the corners in
the non-symmetric Nitsche method, but optimal convergence was still
attained on the finest meshes. We also studied the error in the fluxes
approximated by the multiplier and the results were similar to that of
the $L^2$-norm error.
\begin{figure}
\includegraphics[width=6.5cm]{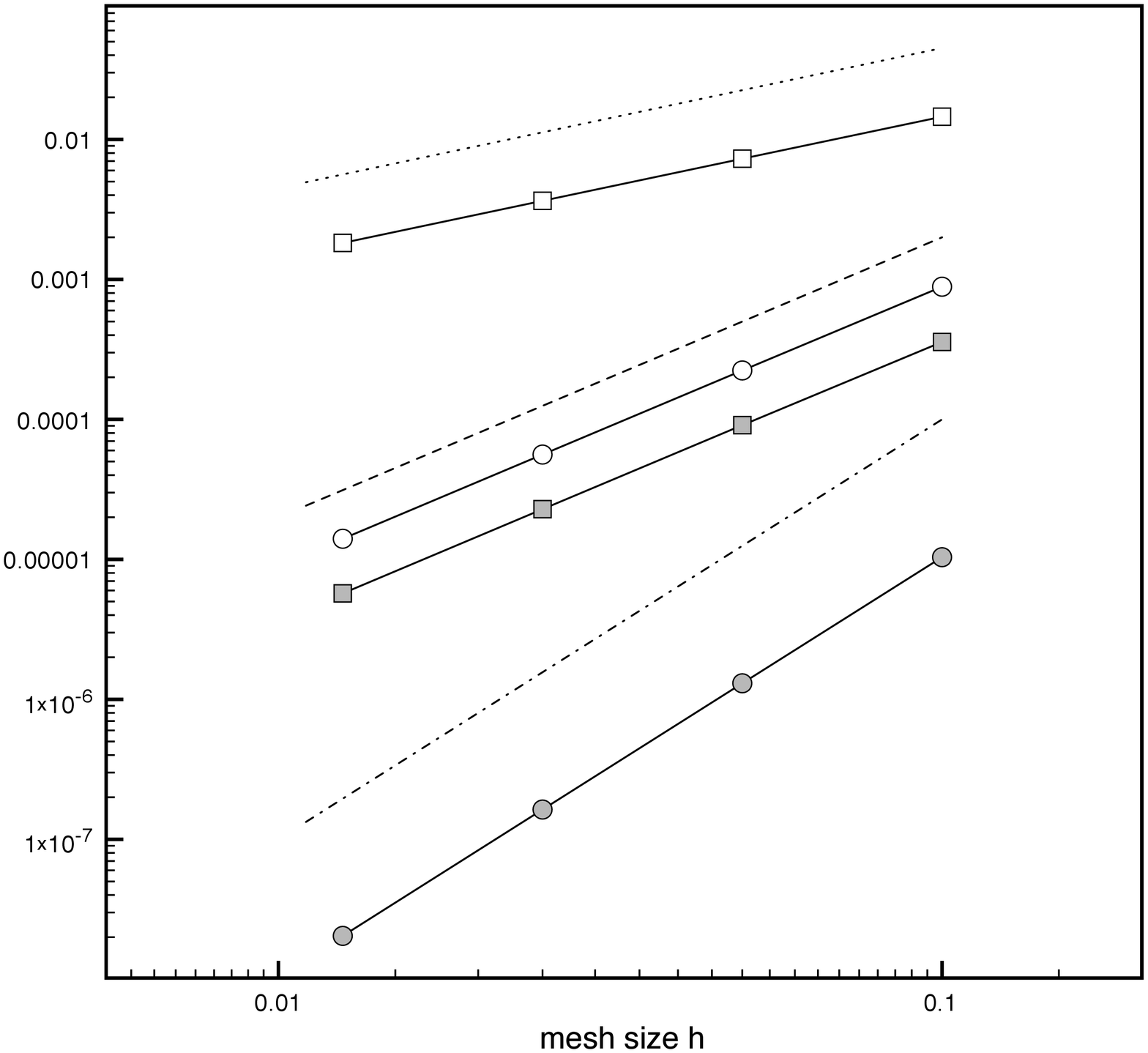}
\includegraphics[width=6.5cm]{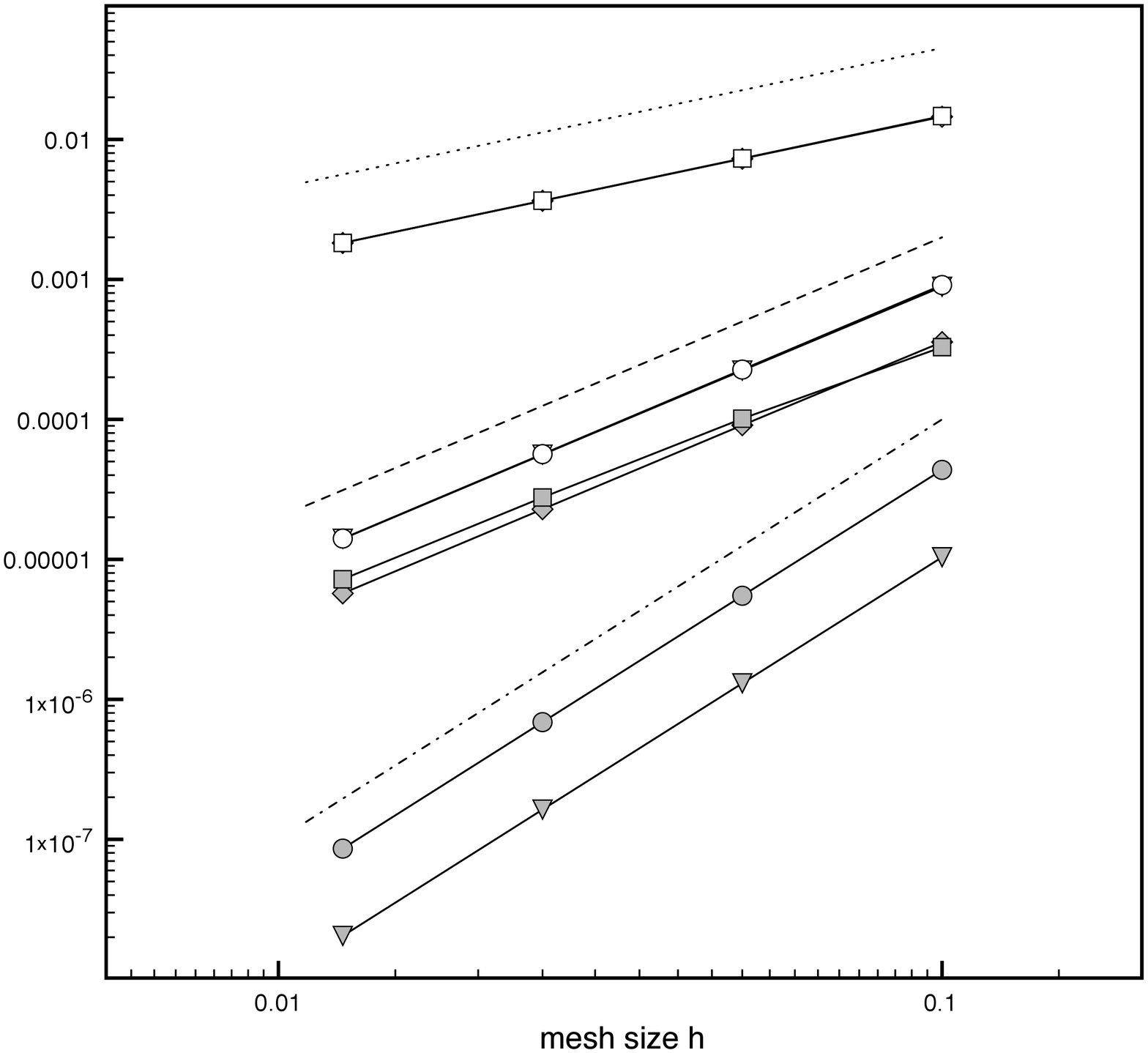}
\caption{Left plot: convergence of the projection stabilised methods,
  ($k=1$ square marker, $k=2$ round marker, markers for $L^2$-error
  filled).
Right plot: convergence of the method \eqref{BB_stab} ($k=1$ square marker, $k=2$ round marker, markers for $L^2$-error filled) and the stable Lagrange multiplier method ($k=1$ diamond marker, $k=2$ triangular marker, markers for $L^2$-error  filled).  }
\label{fig:proj_stab}
\end{figure}
\begin{figure}
\includegraphics[width=6.5cm]{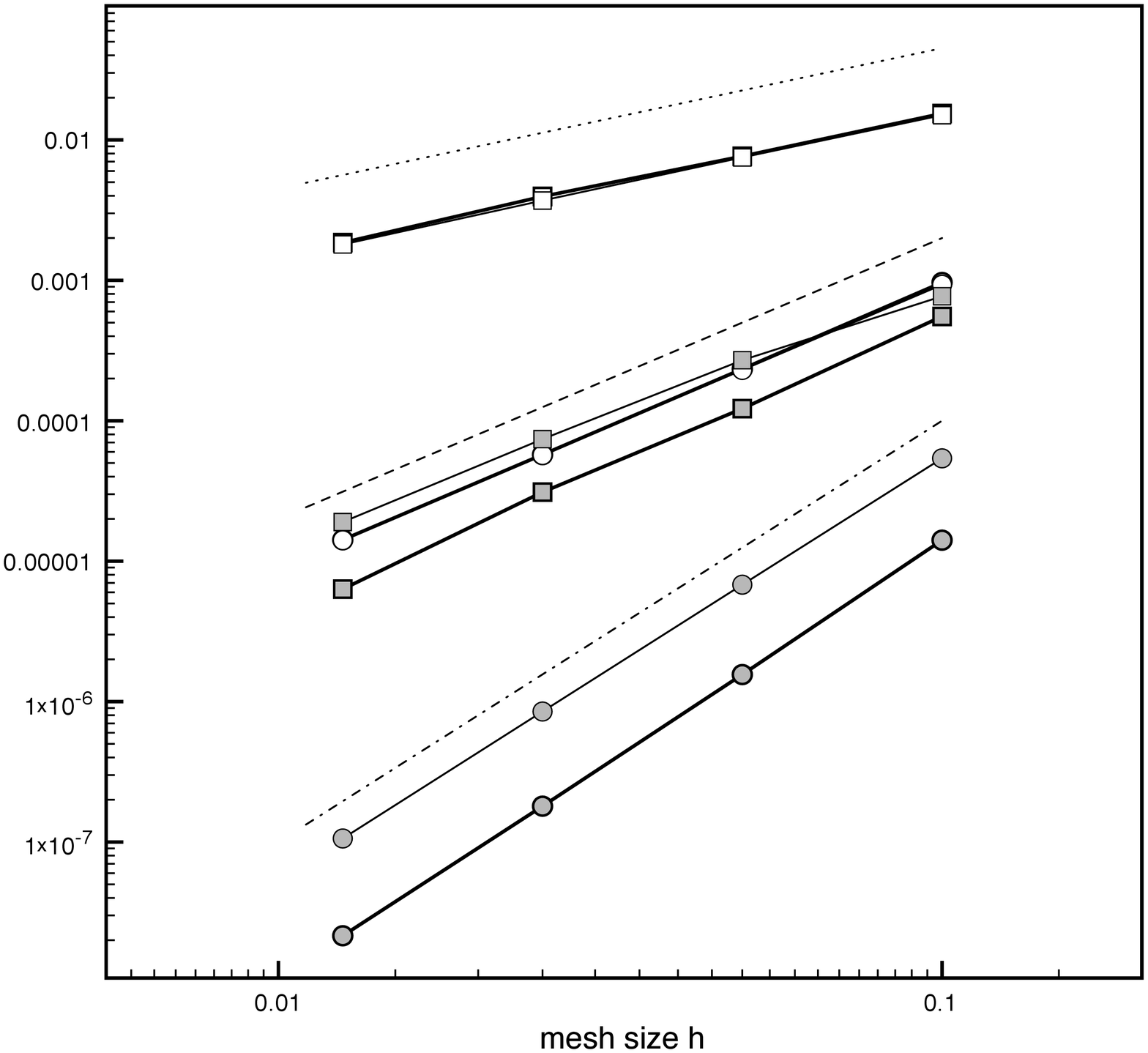}
\includegraphics[width=6.5cm]{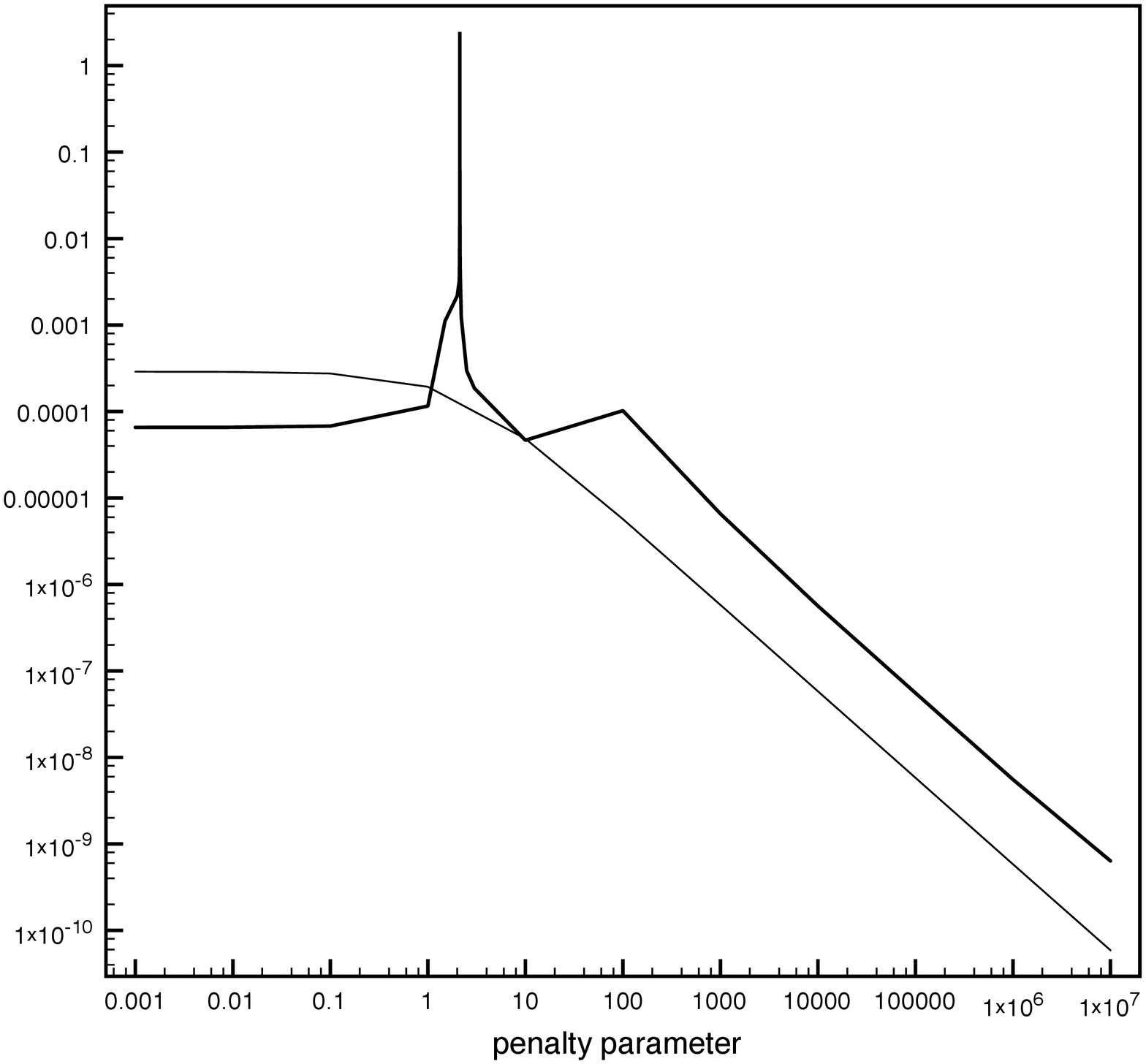}
\caption{Left plot: convergence of the penalty free Nitsche type methods,
  ($k=1$ square marker, $k=2$ round marker, markers for $L^2$-error
  filled, symmetric version plotted with thick line).
Right plot: asymptotic behavior of the difference in the $L^2$-norm
between the solution of the Barbosa-Hughes method and the
corresponding Nitsche type method (symmetric version plotted with thick line).  }
\label{fig:Nitsche}
\end{figure}
\section{CONCLUSION}
We have given an analysis of projection stabilised Lagrange
multipliers in an abstract framework and shown some applications of
this therory. We then showed how the residual based stabilisation
method of Barbosa-Hughes can be
interpreted as a method penalising the distance to a stable subspace
by relating it to the inf-sup stable penalty free Nitsche method. The
methods were tested and compared numerically on a simple model
problem. All these methods appear to have very similar properties. In
particular optimal convergence was observed in both the $H^1$- and the
$L^2$-norms independently of adjoint consistency. Nevertheless adjoint consistent
methods have smaller errors in $L^2$-norm for a fixed mesh size and
similarly for the approximation of the fluxes. The observed difference was a moderate factor. One may
conclude from this that it is 
reasonable that one may base the choice of method entirely on what is
the easiest to implement for a given application. It also follows from
the discussion of Section \ref{projpenal} that the jump penalty operator
provides a stabilisation requiring minimal knowledge of the inf-sup
stable space $L_h$.
Two open problems
are the question of stability of the penalty free symmetric Nitsche
method and accuracy in the $L^2$-norm of the non-symmetric Nitsche
method.
Both of which are observed.\\[5mm]
\begin{center}
{\bf{Acknowledgment}}
\end{center}
Section \ref{intro} to \ref{projpenal} of this paper was written for a doctoral course given
in September 2009 at the doctoral school ICMS, Paris-Est
Marne-la-Vall\'ee, that I gave as invited Professor. The kind
hospitality of Professors Alexandre Ern and Robert Eymard is
graciously acknowledged. I also acknowledge funding from EPSRC (award
number EP/J002313/1). Finally I would like to thank the reviewers of
the paper whose constructive criticism helped improve the manuscript.
\begin{center}
{\bf Appendix: construction of the Fortin interpolant}
\end{center}
We will use the notation of section \ref{sec:applications}\ref{unfitted} and prove that the
Fortin interpolant $\pi_F v$ satisfying \eqref{fortin} exists and that
the stability constant is independent of how the interface $\Gamma$
cuts the mesh $\mathcal{T}_h.$ Note that the stability of the Fortin
interpolant writes
\[
\sum_{i=1}^2 \|\nabla \pi_F v\|_{\Omega_i}^2 + \|\pi_F v_1 - \pi_F
v_2\|^2_{\frac12,\Gamma}\lesssim \sum_{i=1}^2 \|\nabla v\|_{\Omega_i}^2 + \| v_1 - v_2\|^2_{\frac12,\Gamma}
\]
where the hidden constant must be independent of the mesh-interface intersection.
We introduce the extension operators $\mathbb{E}_i$ such that for all
$v \in V_i$,  $\mathbb{E}_i v \in H^1(\mathcal{T}_{ih})$,
$\mathbb{E}_i v\vert_{\Omega_i} = v$ and 
$\|\mathbb{E}_i v\|_{H^1(\mathcal{T}_{ih})} \lesssim
\|v\|_{H^1(\Omega_i)}$. Here $\mathcal{T}_{ih}$ denotes the
mesh-domain defined as $\cup_{K \in \mathcal{T}_{ih}} K$.
Let $\mathcal{I}^i_h:H^1(\mathcal{T}_{ih}) \rightarrow V_{ih}$
denote an $H^1(\Omega)$-stable interpolant. For each $j$ define the
extended patch $\varpi^i_j:= \omega^i_j  \cup F_j$. Then on each
patch $\varpi^i_j $ define a function $\varphi^i_j \in V_{ih}$ with
$\mbox{supp}~ \varphi^i_j = \bar \varpi^i_j $,
$\varphi\vert_{\partial \varpi^i_j \cap \Omega_i} = 0$ and
\[
\int_{F_j} \varphi^i_j ~\mbox{d}s = O(H_i),\, \|\nabla
\varphi^i_j\|_{\varpi^i_j} = O(1).
\]
Define $\pi_F v_i := \mathcal{I}^i_h \mathbb{E}_i v_i + \sum_j \alpha^i_j \varphi^i_j$ where 
$$
\alpha^i_j := \frac{\int_{F_j} (v_i - \mathcal{I}^i_h  \mathbb{E}_i v_i) ~\mbox{d}s}{\int_{F_j} \varphi^i_j ~\mbox{d}s}.
$$
This construction is always possible, provided $H$ is a given (fixed)
factor larger than $h$, typically $H=3h$ is sufficient.
 
Then the orthogonality condition of \eqref{fortin} holds by
construction. It remains to prove the $H^1$-stability. By the triangle
inequality and the disjoint supports of the $\varpi^i_j$ we have,
\begin{equation}\label{H1stab_fortin}
\|\nabla \pi_F v_i\|_{\Omega_i} \lesssim \|\nabla \mathcal{I}^i_h\mathbb{E}_i v_i\|_{\mathcal{T}_{ih}} + \left(\sum_j (\alpha^i_j)^2
\|\nabla \varphi^i_j\|_{\varpi^i_j}^2\right)^{\frac12} = T_1+T_2.
\end{equation}
By the assumed stability of $\mathcal{I}^i_h$ and $\mathbb{E}_i$ we immediately have
\[
T_1 \lesssim  \|\nabla v_i\|_{\Omega_i}.
\]
For $T_2$ we consider one term in the sum and get by the construction 
\begin{multline*}
|\alpha^i_i|
\|\nabla \varphi_j^i\|_{\varpi^i_j}\lesssim H^{-1} \int_{F_j} (v - \mathcal{I}^i_h v)
~\mbox{ds} \lesssim H^{-\frac12} \|v - \mathcal{I}^i_h v\|_{F_j} \\[3mm] \lesssim H^{-1}\|\mathbb{E}_i v -
\mathcal{I}^i_h \mathbb{E}_i v\|_{\varpi^i_j}+  \|\nabla(\mathbb{E}_i v -
\mathcal{I}^i_h \mathbb{E}_i v)\|_{\varpi^i_j}.
\end{multline*}
By the shape regularity of the $\varpi_j^i$ there is no
dependence on the mesh domain intersection in the constants.
Summing over $j$ and using
the fact that the $\varpi^i_j$ are disjoint for fixed $i$ we obtain that
\[
T_2 \lesssim H^{-1} \|\mathbb{E}_i v_i - \mathcal{I}^i_h \mathbb{E}_i v_i\|_{\mathcal{T}_{ih}} + \|\nabla (\mathbb{E}_i v_i - \mathcal{I}^i_h \mathbb{E}_i v_i)\|_{\mathcal{T}_{ih}} 
\] 
and the desired stability estimate follows by the approximation and stability
properties of $\mathcal{I}^i_h$ and the stability of $\mathbb{E}_i$.
It remains to prove that 
\[
\|\pi_F v_1 - \pi_F v_2\|^2_{\frac12,\Gamma}\lesssim \sum_{i=1}^2 \|\nabla v_i\|_{\Omega_i}^2 + \| v_1 - v_2\|^2_{\frac12,\Gamma}.
\]
This follows by adding and subtracting $v_1 - v_2$ in the left
hand side, and using a triangle inequality to obtain
\[
\|\pi_F v_1 - \pi_F v_2\|^2_{\frac12,\Gamma}  \lesssim \| v_1 - v_2\|^2_{\frac12,\Gamma}
+ \|\pi_F v_1 - v_1\|^2_{\frac12,\Gamma}+\|\pi_F v_2 - v_2\|^2_{\frac12,\Gamma}.
\]
We now proceed using a global trace inequality, the stability of
the interpolant $\mathcal{I}^i_h$ and the above bound on the term $T_2$,
to show that 
\[
\|\pi_F v_i - v_i\|^2_{\frac12,\Gamma}\lesssim \|\mathcal{I}^i_h v_i -
v_i\|_{H^1(\Omega_i)}^2 + T_2\lesssim \|\mathcal{I}^i_h \mathbb{E}_i v_i -
\mathbb{E}_i v_i\|_{H^1(\mathcal{T}_{ih})}^2 + \|v\|^2_{V}\lesssim \|v\|^2_{V}.
\]
Collecting the above bounds concludes the proof.

\end{article}
\end{document}